\documentclass{amsart}
\usepackage[utf8]{inputenc}
\usepackage{amsmath}
\usepackage{amssymb}
\usepackage{mathrsfs}
\usepackage{graphicx}

\usepackage[colorlinks]{hyperref}
\usepackage[svgnames, dvipsnames, usenames]{xcolor}
\hypersetup{urlcolor = RedViolet, linkcolor = RoyalBlue, citecolor = ForestGreen}

\usepackage{tikz}
\usepackage{partmac}
\usepackage{graphmor}

\def\ppen{\penalty 300}
\let\col=\colon
\def\colon{\col\ppen}

\theoremstyle{plain} 
\newtheorem{thm}{Theorem}[section]
\newtheorem{prop}[thm]{Proposition}
\newtheorem{lem}[thm]{Lemma}
\newtheorem{cor}[thm]{Corollary}
\theoremstyle{definition}
\newtheorem{defn}[thm]{Definition}
\newtheorem{rem}[thm]{Remark}
\newtheorem{ex}[thm]{Example}
\newtheorem{quest}[thm]{Question}
\newtheorem{conj}[thm]{Conjecture}

\newtheorem{alg}[thm]{Algorithm}
\numberwithin{equation}{section}

\theoremstyle{plain}
\newtheorem{innercustomthm}{Theorem}
\newenvironment{customthm}[1]
  {\renewcommand\theinnercustomthm{#1}\innercustomthm}
  {\endinnercustomthm}

\renewcommand{\theta}{\vartheta}
\renewcommand{\phi}{\varphi}
\renewcommand{\epsilon}{\varepsilon}
\renewcommand{\subset}{\subseteq}
\renewcommand{\supset}{\supseteq}

\newcommand{\N}{\mathbb N}
\newcommand{\Z}{\mathbb Z}

\newcommand{\R}{\mathbb R}
\newcommand{\C}{\mathbb C}

\DeclareMathOperator{\GL}{GL}
\DeclareMathOperator{\Mor}{Mor}

\DeclareMathOperator{\spanlin}{span}
\DeclareMathOperator{\Lrot}{Lrot}
\DeclareMathOperator{\Rrot}{Rrot}
\DeclareMathOperator{\rl}{rl}
\DeclareMathOperator{\id}{id}
\DeclareMathOperator{\Tr}{Tr}

\DeclareMathOperator{\perm}{perm}

\renewcommand{\det}{\mathop{\rm det}\nolimits}
\newcommand{\Cat}{\mathscr{C}}
\newcommand{\Gat}{\mathscr{G}}
\newcommand{\Kat}{\mathscr{K}}
\newcommand{\Dat}{\mathscr{D}}

\newcommand{\nlin}{_{N\mathchar `\-\mathrm{lin}}}
\newcommand{\flin}{_{4\mathchar `\-\mathrm{lin}}}
\newcommand{\KG}{\mathbf{K}}
\newcommand{\HG}{\mathbf{H}}
\newcommand{\T}{\mathcal{T}}
\newcommand{\F}{\mathcal{F}}
\newcommand{\MG}{\mathbf{M}}
\newcommand{\XG}{\mathbf{X}}
\newcommand{\nullG}{\mathbf{0}}
\newcommand{\Mat}{\mathsf{Mat}}
\newcommand{\staralg}{\mathop{\rm\ast\mathchar `\-alg}}

\newcommand{\Lin}{\mathscr{L}}
\newcommand{\Part}{\mathscr{P}}

\newcommand{\GWfour}{\Graph{
\GE 0.3/0.5:1/2.5,2/2.5,3/2.5,4/2.5;
\GE 0.3/0:1/,2/,3/,4/;
\GW 0.5:2.5;
}}

\begin{document}
\title{Free quantum analogue of Coxeter group $D_4$}
\author{Daniel Gromada}
\address{Saarland University, Fachbereich Mathematik, Postfach 151150,
66041 Saarbr\"ucken, Germany}
\email{gromada@math.uni-sb.de}
\date{\today}
\thanks{The author was supported by the collaborative research centre SFB-TRR~195 ``Symbolic Tools in Mathematics and their Application''.}
\thanks{I would like to thank Moritz Weber for giving me the idea to look for quantum analogues of Coxeter groups of type $D$. I thank him also for reading my work and providing valuable comments. I would also like to thank Simon Schmidt for discussions regarding quantum symmetries of demihypercubes and folded hypercubes.}

\begin{abstract}
We define the quantum group $D_4^+$ -- a free quantum version of the demihyperoctahedral group $D_4$ (the smallest representative of the Coxeter series $D$). In order to do so, we construct a free analogue of the property that a $4\times4$ matrix has determinant one. Such analogues of determinants are usually very hard to define for free quantum groups in general and our result only holds for the matrix size $N=4$. The free $D_4^+$ is then defined by imposing this generalized determinant condition on the free hyperoctahedral group $H_4^+$. Moreover, we give a detailed combinatorial description of the representation category of $D_4^+$.
\end{abstract}

\maketitle
\section*{Introduction}
Quantum groups constitute a~generalization of the concept of a~group in non-commutative geometry. In this work, we deal with compact quantum groups as defined by Woronowicz in \cite{Wor87}. Many examples of quantum groups arise by deforming the commutativity relation in the associated Hopf $*$-algebra $O(G)$ (in particular, we have the $q$-deformations such as $U_N^q$, $SU_N^q$, $O_N^q$, $SO_N^q$ and so on), other examples of quantum groups are defined by liberating the commutativity relations. Here the canonical examples are the so-called \emph{free} quantum groups such as the free orthogonal group $O_N^+$, the free symmetric group $S_N^+$ or the free hyperoctahedral group $H_N^+$.

Although we have a~free analogue of $O_N$, we have no free analogue of $SO_N$; although we have the free symmetric group, we have no free alternating group; and although we have the free analogues for the Coxeter series $A$ (the symmetric group), $B$ and $C$ (the hyperoctahedral group), we have so far no free analogue of the Coxeter groups of type $D$. All the mentioned examples, where the definition of a~free counterpart is missing, have one thing in common: Classically, as matrix groups, they can be obtained as normal subgroups imposing some kind of a~determinant condition $\det M=1$ (which indeed always defines a subgroup since determinant is multiplicative; in other words, it is a one dimensional representation). So, the reason what makes it hard to find a~suitable free analogues for these groups is that there is no free determinant. More precisely, there is no non-trivial one-dimensional representation of the free quantum groups such as $O_N^+$, $H_N^+$ or $S_N^+$ (see e.g.\ \cite[Thm.~5.1.1]{Fre17}). Although it is not possible to find some one-dimensional representation of the free quantum groups being the free analogue of the determinant, it still can make sense just to look for a free analogue of the relations $\det=1$.

Let $G$ be a group and $G^+$ a quantum group. We say that $G^+$ is a \emph{liberated version} of $G$ and that $G$ is a \emph{classical version} of $G^+$ if $G$ is a quantum subgroup of $G^+$ given by imposing commutativity in the associated Hopf algebra. Now we would like to say that $G^+$ is a \emph{free} quantum group if $O(G)$ contains no commutativity relations at all. It is not entirely clear, how precisely the formal definition of this statement should look like. In case of homogeneous groups (i.e.\ $S_N\subset G$), a~natural condition is to require that the representation category of $G^+$ can be described by some linear category of non-crossing partitions. In this article, we define a quantum group $D_4^+$, which forms a free analogue of the Coxeter group of type $D$ of size $N=4$ in this above mentioned sense. This is the first main result of this work formulated in Section~\ref{sec.perm}:

\begin{customthm}{A}[Theorem \ref{T.D}]
We have
$$D_4\subsetneq D_4^+\subsetneq H_4^+.$$
The subgroup $D_4$ is obtained from $D_4^+$ imposing commutativity or $(-1)$-commutativity on the entries of the fundamental representation.
\end{customthm}

In Sections \ref{sec.graph} and \ref{sec.finite}, we are looking for a combinatorial description of the representation category of this new quantum group $D_4^+$. In the following theorem, we summarize the main results. The strategy and goals for those two sections are summarized more in detail in the beginning of Section~\ref{sec.graph}. We use the framework of bilabelled graphs recently introduced in \cite{MR19}, see Sect.~\ref{secc.graphcat} for a summary.

\begin{customthm}{B}[Propositions \ref{P.cat}, \ref{P.FCCat}, Theorem \ref{T.Cfin}]
For every $k,l\in\N_0$, let $C(k,l)$ be the set of all bilabelled graphs $\KG=(K,\mathbf{a},\mathbf{b})$ that satisfy the following conditions
\begin{enumerate}
\renewcommand{\theenumi}{\roman{enumi}}
\item $\KG$ is planar,
\item $\tilde d_v$ is even for all $v\in V(K)$, ($\tilde d_v$ is the degree counting also the output strings, see Def.~\ref{D.gendeg})
\item $K$ is bipartite and all the vertices $a_1,\dots,a_k,b_1,\dots,b_l$ are elements of one of the parts.
\item There is no vertex $v$ with $d_v=\tilde d_v=2$.
\item There are no multiple edges.
\item Every component of $\KG$ has at least one vertex among $(\mathbf{a},\mathbf{b})$.
\end{enumerate}
Then the following holds:
\begin{enumerate}
\item The sets $C(k,l)$ are finite for every $k,l$.
\item They model the representation theory of $D_4^+$ through the functor $T^A$, where $A$ is a matrix given by $A_{ij}=\delta_{ij}-1/2$. That is, if we denote by $u$ the fundamental representation of $D_4^+$, then
$$\Mor(u^{\otimes k},u^{\otimes l})=\spanlin\{T_\KG^A\mid \KG\in C(k,l)\}.$$
\end{enumerate}
\end{customthm}

The set $C$ is not closed under the graph category operations introduced in \cite{MR19}. Nevertheless, we can make it into a category by slightly modifying those operations (taking certain quotient essentially). In this way, we obtain a very interesting diagram category, which is worth studying on its own. In Section \ref{sec.open}, we suggest some directions for further research and state some open problems regarding this category.

\section{Preliminaries}
\label{sec.prelim}

In this section we recall the basic notions of compact matrix quantum groups and Tannaka--Krein duality. For a~more detailed introduction, we refer to the monographs \cite{Tim08,NT13}.

\subsection{Compact matrix quantum groups}
\label{secc.qgdef}
A~\emph{compact matrix quantum group} is a pair $G=(A,u)$, where $A$ is a~$*$-algebra and $u=(u_{ij})\in M_N(A)$ is a matrix with values in $A$ such that
\begin{enumerate}
\item the elements $u_{ij}$ $i,j=1,\dots, N$ generate $A$,
\item the matrices $u$ and $u^t=(u_{ji})$ are similar to unitary matrices,
\item the map $\Delta\colon A\to A\otimes A$ defined as $\Delta(u_{ij}):=\sum_{k=1}^N u_{ik}\otimes u_{kj}$ extends to a~$*$-homomorphism.
\end{enumerate}

Compact matrix quantum groups introduced by Woronowicz \cite{Wor87} are generalizations of compact matrix groups in the following sense. For a~matrix group $G\subseteq M_N(\C)$, we define $u_{ij}\colon G\to\C$ to be the coordinate functions $u_{ij}(g):=g_{ij}$. Then we define the \emph{coordinate algebra} $A:=O(G)$ to be the algebra generated by $u_{ij}$. The pair $(A,u)$ then forms a compact matrix quantum group. The so-called \emph{comultiplication} $\Delta\colon O(G)\to O(G)\otimes O(G)$ dualizes matrix multiplication on $G$: $\Delta(f)(g,h)=f(gh)$ for $f\in O(G)$ and $g,h\in G$.

Therefore, for a~general compact matrix quantum group $G=(A,u)$, the algebra~$A$ should be seen as the algebra of non-commutative functions defined on some non-commutative compact underlying space. For this reason, we often denote $A=O(G)$ even if $A$ is not commutative. We can actually define the structure of a~Hopf $*$-algebra on $A$. The matrix $u$ is called the \emph{fundamental representation} of~$G$. Let us note that compact matrix quantum groups are special cases of compact quantum groups, see \cite{NT13,Tim08} for details.

A compact matrix quantum group $H=(O(H),v)$ is a~\emph{quantum subgroup} of $G=(O(G),u)$, denoted as $H\subseteq G$, if $u$ and $v$ have the same size and there is a~surjective $*$-homomorphism $\phi\colon O(G)\to O(H)$ sending $u_{ij}\mapsto v_{ij}$. We say that $G$ and $H$ are \emph{identical} if there exists such a $*$-isomorphism (i.e. if $G\subset H$ and $H\subset G$). If there is a~$*$-isomorphism mapping $u_{ij}\mapsto [TvT^{-1}]_{ij}$ for some invertible matrix $T\in\GL_N$, we say that $G$ is similar to $H$ and write $G=THT^{-1}$.

One of the most important examples is the quantum generalization of the orthogonal group -- the \emph{free orthogonal quantum group} defined by Wang in \cite{Wan95free} through the universal $*$-algebra
$$O(O_N^+):=\staralg(u_{ij},\;i,j=1,\dots,N\mid u_{ij}=u^*_{ij},uu^t=u^tu=1_N).$$

In this article, we are going to study only \emph{orthogonal} quantum groups $G\subset O_N^+$. In other words, we will always assume that the fundamental representation $u$ of $G$ is orthogonal (satisfies $u_{ij}=u_{ij}^*$ and $uu^t=u^tu=1$).

For a~compact matrix quantum group $G=(O(G),u)$, we say that $v\in M_n(O(G))$ is a~representation of $G$ if $\Delta(v_{ij})=\sum_{k}v_{ik}\otimes v_{kj}$, where $\Delta$ is the comultiplication. The representation $v$ is called \emph{unitary} if it is unitary as a~matrix, i.e. $\sum_k v_{ik}v_{jk}^*=\sum_k v_{ki}^*v_{kj}=\delta_{ij}$. In particular, an element $a\in O(G)$ is a one-dimensional representation if $\Delta(a)=a\otimes a$. Another example of a quantum group repersentation is the fundamental representation $u$.

For two representations $v\in M_n(O(G))$, $w\in M_m(O(G))$ of $G$ we define the space of \emph{intertwineres}
$$\Mor(v,w)=\{T\colon \C^n\to\C^m\mid Tv=wT\}.$$

\subsection{Deformed commutativity}
\label{secc.qcom}

Although the goal of this article is to construct some \emph{free} quantum group, that is, to \emph{liberate} the commutativity relations, we also need to say something about \emph{deformations} of commutativity -- namely the $q$-commutativity at $q=-1$. Here, the situation is much closer to the commutative setting than in the free case. For example, we have the quantum determinant here. (See \cite[Section 9.2]{KS97} for more details.)

We say that a~matrix $u$ has \emph{$(-1)$-commutative} entries if the following relations hold
$$u_{ik}u_{jk}=-u_{jk}u_{ik},\quad u_{ki}u_{kj}=-u_{kj}u_{ki},\quad u_{ik}u_{jl}=u_{jl}u_{ik}$$
assuming $i\neq j$ and $k\neq l$.

As an example, let us mention the $q=-1$ deformation of the orthogonal group denoted by $O_N^{-1}$, which is defined by
$$O(O_N^{-1})=\staralg(u_{ij}\mid\text{$u=\bar u$, $u$ orthogonal, $u$ $(-1)$-commutative}).$$

For $q$-commutative matrices, we define the \emph{$q$-determinant}. The definition for $q=-1$ is as follows
\begin{align*}
\det_{-1}u&=\sum_{\sigma\in S_N}u_{1\sigma(1)}\cdots u_{N\sigma(N)}=\sum_{\sigma\in S_N}u_{\pi(1)\sigma(1)}\cdots u_{\pi(N)\sigma(N)}\\
&=\sum_{\pi\in S_N}u_{\pi(1)1}\cdots u_{\pi(N)N}=\sum_{\pi\in S_N}u_{\pi(1)\sigma(1)}\cdots u_{\pi(N)\sigma(N)}
\end{align*}
Since we assume the $(-1)$-commutativity, all the factors of the terms always mutually commute. From this, the different ways to write down the determinant follow. The definition is very similar to the one of classical determinant --  we are only missing the sign of the permutation in the sum. Such an object is sometimes considered also in the classical theory of matrices, where it is called the \emph{permanent}.

\begin{prop}
\label{P.detrep}
Let $G=(O(G),u)$ be a~compact matrix quantum group such that $u$ has $(-1)$-commutative entries. Then $\det_{-1}u\in O(G)$ forms a~one-dimensional representation of $G$.
\end{prop}
\begin{proof}
Indeed,
\begin{align*}
&\Delta(\det_{-1}u)=\sum_{\sigma\in S_N}\sum_{k_1,\dots,k_N=1}^N(u_{1k_1}\cdots u_{Nk_N})\otimes(u_{k_1\sigma(1)}\dots u_{k_N\sigma(N)})\\&=\sum_{\sigma\in S_N}\sum_{\pi\in S_N}(u_{1\pi(1)}\cdots u_{N\pi(N)})\otimes(u_{\pi(1)\sigma(1)}\dots u_{\pi(N)\sigma(N)})=\det_{-1}u\otimes\det_{-1}u,
\end{align*}
where the second equality follows from the fact that only tuples $(k_1,\dots,k_N)$ that are permutations contribute to the sum since if $k_i=k_j$ then $u_{k_1\sigma(1)}\cdots u_{k_N\sigma(N)}=-u_{k_1(\sigma\circ\tau)(1)}\cdots u_{k_N(\sigma\circ\tau)(N)}$, where $\tau$ is the transposition of $i$ and $j$.
\end{proof}

\subsection{Monoidal involutive categories and Tannaka--Krein duality}
All categories appearing in this article will be rigid monoidal involutive categories with the set of natural numbers including zero $\N_0$ as the set of self-dual objects. That is, by a~\emph{category} $\Cat$ we mean a~collection of sets $\Cat(k,l)$ with $k,l\in\N_0$ together with a~colection of operations
\begin{itemize}
\item tensor product $\otimes\colon \Cat(k_1,l_1)\times\Cat(k_2,l_2)\to\Cat(k_1+k_2,l_1+l_2)$,
\item composition $\cdot\colon \Cat(l,m)\times\Cat(k,l)\to\Cat(k,m)$,
\item involution $*\colon \Cat(k,l)\to\Cat(l,k)$
\end{itemize}
such that
\begin{itemize}
\item $\otimes$ and $\cdot$ are associative, $*$ is involutive
\item we have $(R\otimes T)(S\otimes U)=(RS\otimes TU)$ for every $R\in\Cat(l_1,m_1)$, $S\in\Cat(k_1,l_1)$, $T\in\Cat(l_2,m_2)$, $U\in\Cat(k_2,l_2)$,
\item there is the \emph{scalar identity} $\id_0\in\Cat(0,0)$ such that $T\otimes \id_0=\id_0\otimes T=T$ for every $T\in\Cat(k,l)$,
\item there is the \emph{identity} $\id\in\Cat(1,1)$ such that $T\cdot\id_k=\id_l\cdot T=T$ for every $T\in\Cat(k,l)$, where $\id_k:=\id^{\otimes k}$ for $k\ge 1$
\item there is the \emph{duality morphism} $T_{\pairpart}\in\Cat(0,2)$ such that
$$(\id\otimes T_{\pairpart}^*)(T_{\pairpart}\otimes\id)=\id\quad\text{and}\quad(T_{\pairpart}^*\otimes\id)(\id\otimes T_{\pairpart})=\id.$$
\end{itemize}
By a~{\em linear category} we mean a~category in the above sense such that the sets $\Cat(k,l)$ are vector spaces, the operations $\otimes$ and $\cdot$ are bilinear and the involution $*$ is antilinear.

The role of homomorphisms between two categories is played by monoidal unitary functors. A~\emph{monoidal unitary functor} $F\colon\Cat\to\Dat$ is a~collection of maps $F\colon\Cat(k,l)\to\Dat(k,l)$ such that $F(ST)=F(S)F(T)$ and $F(T^*)=F(T)^*$ for every $T\in\Cat(k,l)$ and $S\in\Cat(l,m)$. We call $F$ a~category \emph{isomorphism} if all the maps are bijections.

\begin{rem}[Frobenius reciprocity]
The duality morphism defines a~linear isomorphism between the spaces $\Cat(k,l)$ with fixed $k+l$. For a~morphism $T\in\Cat(k,l)$, we define its \emph{right rotation} and \emph{left rotation} as
\begin{align*}
\Rrot T&:=(\id_{l-1}\otimes T_{\pairpart}^*)(T\otimes\id)\in\Cat(k+1,l-1),\\
\Lrot T&:=(\id\otimes T)(T_{\pairpart}\otimes\id_{k-1})\in\Cat(k-1,l+1).
\end{align*}
In particular, any category is determined by the spaces $\Cat(0,l)$ since then we can compute $\Cat(k,l)=\Rrot^k\Cat(0,k+l)$.
\end{rem}

Let $G=(O(G),u)$ be a~compact matrix quantum group. We can associate to $G$ the linear category $\Cat_G$ with
$$\Cat_G(k,l):=\Mor(u^{\otimes k},u^{\otimes l})=\{T\colon(\C^N)^{\otimes k}\to(\C^N)^{\otimes l}\mid Tu^{\otimes k}=u^{\otimes l}T\}.$$

Another examples of monoidal categories will be presented in sections \ref{secc.partcat}, \ref{secc.graphcat} with the aim to model somehow such representation categories. Now the key ingredient for our work is the so-called Tannaka--Krein duality formulated by Woronowicz in \cite{Wor88}. We give here a simplified version of this statement (see also \cite{NT13,Mal18,Fre19survey}).

\begin{thm}[Tannaka--Krein duality]
\label{T.Tannaka}
Let $\Cat$ be a~linear category with $\Cat(k,l)\subset\Lin((\C^N)^{\otimes k},(\C^N)^{\otimes l})$ for some $N\in\N$. Suppose that the duality morphism $T_{\pairpart}\colon\C\to(\C^N)^{\otimes 2}$ is of the form $1\mapsto \sum_{i=1}^Ne_i\otimes e_i$. Then there exists a~compact matrix quantum group $G=(O(G),u)\subset O^+_N$ such that $\Cat=\Cat_G$.
\end{thm}

Finally, let us mention a well known correspondence between quantum subgroups and subcategories.

\begin{prop}
\label{P.subgrp}
Consider $G,H\subset O_N^+$. Then $H\subset G$ if and only if $\Cat_H(k,l)\supset\Cat_G(k,l)$ for every $k,l\in\N_0$.
\end{prop}

Considering two quantum groups $G_1$ and $G_2$, we may define their \emph{intersection} $G:=G_1\cap G_2$ to be the largest quantum subgroup of both $G_1$ and $G_2$. The associated representation category $\Cat_G$ would then consequently be the smallest category containing both $\Cat_{G_1}$ and $\Cat_{G_2}$. We may say that $\Cat_G$ is \emph{generated} by $\Cat_{G_1}$ and $\Cat_{G_2}$.

\subsection{Partitions}

Let $k,l\in\N_0$, by a~\emph{partition} of $k$ upper and $l$ lower points we mean a~partition of the set $\{1,\dots,k\}\sqcup\{1,\dots,l\}\approx\{1,\dots,k+l\}$, that is, a~decomposition of the set of $k+l$ points into non-empty disjoint subsets, called \emph{blocks}. The first $k$ points are called \emph{upper} and the last $l$ points are called \emph{lower}. The set of all partitions on $k$ upper and $l$ lower points is denoted $\Part(k,l)$. We define the union $\Part:=\bigcup_{k,l\in\N_0}\Part(k,l)$. The number $\left| p\right|:=k+l$ for $p\in\Part(k,l)$ is called the \emph{length} of $p$.

We illustrate partitions graphically by putting $k$ points in one row and $l$ points on another row below and connecting by lines those points that are grouped in one block. All lines are drawn between those two rows.

Below, we give an example of two partitions $p\in \Part(3,4)$ and $q\in\Part(4,4)$ defined by their graphical representation. The first set of points is decomposed into three blocks, whereas the second one is into five blocks. In addition, the first one is an example of a~\emph{non-crossing} partition, i.e.\ a~partition that can be drawn in a~way that lines connecting different blocks do not intersect (following the rule that all lines are between the two rows of points). On the other hand, the second partition has one crossing. 

\begin{equation}
\label{eq.pq}
\vrule height 16bp depth 10bp width 0bp
p=
\BigPartition{
\Pblock 0 to 0.25:2,3
\Pblock 1 to 0.75:1,2,3
\Psingletons 0 to 0.25:1,4
\Pline (2.5,0.25) (2.5,0.75)
}
\qquad
q=
\BigPartition{
\Psingletons 0 to 0.25:1,4
\Psingletons 1 to 0.75:1,4
\Pline (2,0) (3,1)
\Pline (3,0) (2,1)
\Pline (2.75,0.25) (4,0.25)
}
\end{equation}


A block containing a~single point is called a~\emph{singleton}. In particular, the partitions containing only one point are called singletons and for clarity denoted by an arrow $\singleton\in\Part(0,1)$ and $\upsingleton\in\Part(1,0)$.

\subsection{Linear categories of partitions}
\label{secc.partcat}
Below, we define operations on partitions that give the collection of sets $\Part(k,l)$ the structure of a~(monoidal involutive) category. These operations probably first appeared in \cite{Mar96}. In the context of quantum groups, the partition categories were first considered in \cite{BS09}. Note that in contrast with \cite{BS09}, we consider linear categories of partitions here, that is, we also allow linear combinations of partitions; see also \cite{GWintsp,GWgen}.

Fix a~natural number $N\in\N$. Let us denote $\Part\nlin(k,l)$ the vector space of formal linear combination of partitions $p\in\Part(k,l)$. That is, $\Part\nlin(k,l)$ is a~vector space, whose basis is $\Part(k,l)$. Let us denote $\Part\nlin:=\bigcup_{k,l}\Part\nlin(k,l)$.

Now, we are going to define some operations on $\Part\nlin$. First, let us define those operations just on partitions.
\begin{itemize}
\item  The \emph{tensor product} of two partitions $p\in\Part(k,l)$ and $q\in\Part(k',l')$ is the partition $p\otimes q\in \Part(k+k',l+l')$ obtained by writing the graphical representations of $p$ and $q$ ``side by side''.
$$
\BigPartition{
\Pblock 0 to 0.25:2,3
\Pblock 1 to 0.75:1,2,3
\Psingletons 0 to 0.25:1,4
\Pline (2.5,0.25) (2.5,0.75)
}
\otimes
\BigPartition{
\Psingletons 0 to 0.25:1,4
\Psingletons 1 to 0.75:1,4
\Pline (2,0) (3,1)
\Pline (3,0) (2,1)
\Pline (2.75,0.25) (4,0.25)
}
=
\BigPartition{
\Pblock 0 to 0.25:2,3
\Pblock 1 to 0.75:1,2,3
\Psingletons 0 to 0.25:1,4,5,8
\Psingletons 1 to 0.75:5,8
\Pline (2.5,0.25) (2.5,0.75)
\Pline (6,0) (7,1)
\Pline (7,0) (6,1)
\Pline (6.75,0.25) (8,0.25)
}
$$

\item For $p\in\Part(k,l)$, $q\in\Part(l,m)$ we define their \emph{composition} $qp\in\Part\nlin(k,m)$ by putting the graphical representation of $q$ below $p$ identifying the lower row of $p$ with the upper row of $q$. The upper row of $p$ now represents the upper row of the composition and the lower row of $q$ represents the lower row of the composition. Each extra connected component of the diagram that appears in the middle and is not connected to any of the upper or the lower points, transforms to a~multiplicative factor $N$.
$$
\BigPartition{
\Psingletons 0 to 0.25:1,4
\Psingletons 1 to 0.75:1,4
\Pline (2,0) (3,1)
\Pline (3,0) (2,1)
\Pline (2.75,0.25) (4,0.25)
}
\cdot
\BigPartition{
\Pblock 0 to 0.25:2,3
\Pblock 1 to 0.75:1,2,3
\Psingletons 0 to 0.25:1,4
\Pline (2.5,0.25) (2.5,0.75)
}
=
\BigPartition{
\Pblock 0.5 to 0.75:2,3
\Pblock 1.5 to 1.25:1,2,3
\Psingletons  0.5 to  0.75:1,4
\Pline (2.5,0.75) (2.5,1.25)
\Psingletons -0.5 to -0.25:1,4
\Psingletons  0.5 to  0.25:1,4
\Pline (2,-0.5) (3,0.5)
\Pline (3,-0.5) (2,0.5)
\Pline (2.75,-0.25) (4,-0.25)
}
= N^2
\BigPartition{
\Pblock 0 to 0.25:2,3,4
\Pblock 1 to 0.75:1,2,3
\Psingletons 0 to 0.25:1
\Pline (2.5,0.25) (2.5,0.75)
}
$$

\item For $p\in\Part(k,l)$ we define its \emph{involution} $p^*\in\Part(l,k)$ by reversing its graphical representation with respect to the horizontal axis.
$$
\left(
\BigPartition{
\Pblock 0 to 0.25:2,3
\Pblock 1 to 0.75:1,2,3
\Psingletons 0 to 0.25:1,4
\Pline (2.5,0.25) (2.5,0.75)
}
\right)^*
=
\BigPartition{
\Pblock 1 to 0.75:2,3
\Pblock 0 to 0.25:1,2,3
\Psingletons 1 to 0.75:1,4
\Pline (2.5,0.25) (2.5,0.75)
}
$$
\end{itemize}

Now we can extend the definition of tensor product and composition linearly. We extend the definition of the involution antilinearly. Then the operations define the structure of a~linear category on the collection $\Part\nlin(k,l)$.

Any collection of subspaces $\Kat=\bigcup_{k,l\in\N_0}\Kat(k,l)$, $\Kat(k,l)\subset\Part\nlin(k,l)$ containing the \emph{identity partition} $\idpart\in\Kat(1,1)$ and the \emph{pair partition} $\pairpart\in\Kat(0,2)$ (playing the role of the duality morphism) and closed under the category operations is again a~linear category called a~\emph{linear category of partitions}.

Note that for any linear category of partitions, the space $\Cat(0,0)$ is one-dimensional, formed by scalar multiples of the \emph{empty partition} playing the role of the scalar identity and hence can be identified with $\C$.

For given $p_1,\dots,p_n\in\Part\nlin$, we denote by $\langle p_1,\dots,p_n\rangle\nlin$ the smallest linear category of partitions containing $p_1,\dots,p_n$. We say that $p_1,\dots,p_n$ \emph{generate} $\langle p_1,\dots,p_n\rangle\nlin$. Note that the pair partitions are contained in the category by definition and hence will not be explicitly listed as generators.

For any partition, its right and left rotations are given by simply rotating the picture. That is, right rotation is given by taking the rightmost point of the lower row and putting it to the top row. For the left rotation, take the leftmost point of the upper row and put it to the bottom row. For example,
$$
\Rrot\left(
\BigPartition{
\Pblock 0 to 0.25:2,3
\Pblock 1 to 0.75:1,2,3
\Psingletons 0 to 0.25:1,4
\Pline (2.5,0.25) (2.5,0.75)
}\right)=
\BigPartition{
\Pblock 0 to 0.25:2,3
\Pblock 1 to 0.75:1,2,3
\Psingletons 1 to 0.75:4
\Psingletons 0 to 0.25:1
\Pline (2.5,0.25) (2.5,0.75)
},\qquad
\Lrot\left(
\BigPartition{
\Pblock 0 to 0.25:2,3
\Pblock 1 to 0.75:1,2,3
\Psingletons 0 to 0.25:1,4
\Pline (2.5,0.25) (2.5,0.75)
}\right)=
\BigPartition{
\Pblock 0 to 0.35:0,2,3
\Pblock 1 to 0.75:2,3
\Psingletons 0 to 0.2:1,4
\Pline (2.5,0.35) (2.5,0.75)
}.
$$

\begin{rem}
As we already mentioned, Banica and Speicher in \cite{BS09} and many others following their work do not use the linear structure for partitions. Any category of partitions $\Cat$ in the Banica--Speicher sense defines a linear category of partitions $\Kat$ by taking $\Kat(k,l):=\spanlin\Cat(k,l)$. Such categories are much easier to study by combinatorial means and therefore we call them \emph{easy}. Many examples of categories that are not of this form (and hence called \emph{non-easy}) were recently discovered in \cite{GWintsp,GWgen}.
\end{rem}

\subsection{Linear maps associated to partitions} In this section, we review the connection between partitions and quantum groups \cite{BS09}. More concretely, we show how partition categories can be used to model representation categories of quantum groups.

Consider again a~fixed natural number $N\in\N$. Given a~partition $p\in\Part(k,l)$, we can define a~linear map $T_p\colon(\C^N)^{\otimes k}\to(\C^N)^{\otimes l}$ via
\begin{equation}
T_p(e_{i_1}\otimes\cdots\otimes e_{i_k})=\sum_{j_1,\dots,j_l=1}^N\delta_p(\mathbf{i},\mathbf{j})(e_{j_1}\otimes\cdots\otimes e_{j_l}),
\end{equation}
where $\mathbf{i}=(i_1,\dots,i_k)$, $\mathbf{j}=(j_1,\dots,j_l)$ and the symbol $\delta_p(\mathbf{i},\mathbf{j})$ is defined as follows. Let us assign the $k$ points in the upper row of $p$ by numbers $i_1,\dots,i_k$ (from left to right) and the $l$ points in the lower row $j_1,\dots,j_l$ (again from left to right). Then $\delta(\mathbf{i},\mathbf{j})=1$ if the points belonging to the same block are assigned the same numbers. Otherwise $\delta(\mathbf{i},\mathbf{j})=0$. 

As an example, we can express $\delta_p$ and $\delta_q$, where $p$ and $q$ come from Equation \eqref{eq.pq}, using multivariate $\delta$ function as follows
$$\delta_p(\mathbf{i},\mathbf{j})=\delta_{i_1i_2i_3j_2j_3},\quad
\delta_q(\mathbf{i},\mathbf{j})=\delta_{i_2j_3j_4}\delta_{i_3j_2}.$$

We extend this definition for linear combinations of partitions linearly, that is, $\delta_{\alpha p+q}=\alpha\delta_p+\delta_q$ and hence $T_{\alpha p+q}=\alpha T_p+T_q$.

Given a~linear combination of partitions $p\in\Part\nlin(k,l)$, we can interpret the map $T_p$ as an intertwiner $T_pu^{\otimes k}=u^{\otimes l}T_p$ for some compact matrix quantum group~$G$. Substituting the definition of $T_p$, this implies the following relations
$$\sum_{t_1,\dots,t_k=1}^N\delta_p(\mathbf{t},\mathbf{s})u_{t_1i_1}\cdots u_{t_ki_k}=\sum_{j_1,\dots,j_l=1}^N\delta_p(\mathbf{i},\mathbf{j})u_{s_1j_1}\cdots u_{s_lj_l}$$
for every $i_1,\dots,i_k,s_1,\dots,s_l\in\{1,\dots,N\}$.

For example, considering $p=\pairpart\in\Part(0,2)$, we have the relation $$\delta_{s_1s_2}=\sum_{j=1}^Nu_{s_1j}u_{s_2j}.$$ Thus, for any quantum group $G\subset O_N^+$, we have that $T_{\pairpart}\in\Mor(1,u\otimes u)$. Similarly, we also have $T_{\uppairpart}\in\Mor(u\otimes u,1)$ for any $G\subset O_N^+$.

\begin{prop}
The map $T_\bullet\colon p\mapsto T_p$ is a~monoidal unitary functor. That is, we have the following
\begin{enumerate}
\item $T_{p\otimes q}=T_p\otimes T_q$,
\item $T_{qp}=T_qT_p$ whenever one of the sides makes sense,
\item $T_{p^*}=T_p^*$.
\end{enumerate}
\end{prop}

\begin{cor}
For every linear category of partitions $\Kat$ there exists a~compact matrix quantum group $G=(O(G),u)\subset O_N^+$ such that
$$\Cat_G(k,l)=\Mor(u^{\otimes k},u^{\otimes l})=\{T_p\mid p\in\Kat(k,l)\}.$$
\end{cor}
\begin{proof}
Follows by applying the Tannaka--Krein duality to the image of $\Kat$ by the functor $p\mapsto T_p$.
\end{proof}

In addition, we can express the associated quantum group $G$ very concretely by considering the universal $*$-algebra given by the intertwiner relations $T_pu^{\otimes k}=u^{\otimes l}T_p$. That is,
$$O(G)=\staralg(u_{ij},\;i,j=1,\dots,N\mid u=\bar u,\; T_pu^{\otimes k}=u^{\otimes l}T_p\;\forall p\in\Kat)$$

\begin{rem}
\label{R.notinj}
The functor $p\mapsto T_p$ is not injective. Consequently, given a~linear category of partitions $\Kat$ and the associated quantum group $G$, there may exist a~different partition category $\Kat'$ that corresponds to the same quantum group $G$. In particular, having $p\not\in\Kat$ does not imply that $T_p\not\in\Cat_G$.
\end{rem}

\begin{lem}
\label{L.strict}
Consider linear categories of partitions $\Kat_1$, $\Kat_2$ and the associated quantum groups $G_1$ and $G_2$. If $\Kat_1(k,l)\neq\Kat_2(k,l)$ for some $k,l\in\N_0$ with $k+l\le N$, then $G_1\neq G_2$.
\end{lem}
\begin{proof}
It is obvious that two identical quantum groups must have identical representation categories. However, as was mentioned in Remark~\ref{R.notinj}, there may exist two distinct categories of partitions that model the same quantum group since the functor $p\mapsto T_p$ is not injective.

Nevertheless, in \cite[Corollary~3.4]{GWintsp}, it was proven that $p\mapsto T_p$ acts injectively on all the morphism spaces $\Part\nlin(k,l)$ with $k+l\le N$. Therefore, if $\Kat_1(k,l)\neq\Kat_2(k,l)$ for $k+l\le N$, their images under $T_\bullet$ must be distinct and hence also the quantum groups are distinct.
\end{proof}

\subsection{Important examples of partition quantum groups}
\label{secc.qgex}

In this subsection, we summarize some important partition categories and the corresponding quantum groups.

First of all, note that the relation $T_pu^{\otimes k}=u^{\otimes l}T_p$ corresponding to the \emph{crossing partition} $p=\crosspart$ reads $u_{ij}u_{kl}=u_{kl}u_{ij}$, i.e.\ stands for commutativity. This means that the quantum group $G$ corresponding to a~given partition category $\Kat$ is actually a~group if $\crosspart\in\Kat$. Such a~category is also called a~\emph{group category}. Let us mention a~few important examples \cite{BS09}.

\begin{itemize}
\item The category $\Part\nlin=\langle\crosspart,\fourpart,\singleton\rangle\nlin$ of all partitions corresponds to the symmetric group $S_N$ represented by permutation matrices.
\item The category $\langle\crosspart,\fourpart\rangle\nlin$ spanned by all partitions with blocks of even size corresponds to the hyperoctahedral group $H_N$ represented by signed permutation matrices.
\item The category $\langle\crosspart\rangle\nlin$ spanned by all pair partitions (all blocks have size two) corresponds to the orthogonal group $O_N$.
\end{itemize}

In general, partitions containing some crossing always correspond to relations containing some commutation. Following this idea, \emph{free quantum groups} are those that correspond to categories spanned by non-crossing partitions only. The set of all non-crossing partitions itself forms a~category $NC\nlin$. As a~consequence, for any category $\Kat$, we can define a~new non-crossing category $\Kat^+:=\Kat\cap NC\nlin$ (taking $\Kat^+(k,l)=\Kat(k,l)\cap NC\nlin(k,l)$). Using this construction, we can define free counterparts of the group examples above \cite{BS09}.

\begin{itemize}
\item The category $NC\nlin=\langle\fourpart,\singleton\rangle\nlin$ of all non-crossing partitions corresponds to the free symmetric quantum group $S_N^+$.
\item The category $\langle\fourpart\rangle\nlin$ spanned by all non-crossing partitions with blocks of even size corresponds to the free hyperoctahedral quantum group $H_N^+$.
\item The smallest possible category $\langle\rangle\nlin$ spanned by all non-crossing pair partitions corresponds to the free orthogonal quantum group $O_N^+$.
\end{itemize}

We did not define $H_N^+$ and $S_N^+$ in this article. Thanks to Tannaka--Krein duality, one can consider the above given characterization as the definition. Nevertheless, similarly as in the case of $O_N^+$, the quantum groups were actually defined earlier. The definition of $S_N^+$ goes back to \cite{Wan98}, whereas $H_N^+$ comes from \cite{Bic04}. However, many other quantum groups were discovered subsequently by studying other categories of partitions.

\begin{defn}
\label{D.NCversion}
Let $\Kat\subset\Part\nlin$ be a~linear category of partitions containing the crossing partition $\crosspart$. A~category $\Kat^+\subset NC\nlin$ is called a~\emph{non-crossing version} of $\Kat$ if $\langle\Kat^+,\crosspart\rangle\nlin=\Kat$.
\end{defn}

\begin{rem}
One category can have more than one non-crossing version. As an example, take the category $\Kat=\langle\crosspart,\singleton\otimes\singleton\rangle\nlin$. One of its non-crossing versions is obviously $\langle\singleton\otimes\singleton\rangle\nlin$. Another is given by intersection with $NC\nlin$, which can be written as $\Kat\cap NC\nlin=\langle\Labac\rangle\nlin$ and it is different from the previous one \cite{Web13}. Infinitely many new categories interpolating these two were discovered recently \cite{GWext}. The non-crossing version constructed as an intersection with $NC\nlin$ is always the largest one and could be somehow considered as the canonical one. Nevertheless, smaller non-crossing versions may be actually more interesting since the smaller the category, the ``more free'' the quantum group.
\end{rem}

\begin{rem}
If we are in the ``easy case'', i.e.\ if we work only with partitions and not their linear combinations, then every group category $\Kat$ has its non-crossing version -- namely the intersection $\Kat\cap NC\nlin$. This might not hold when working with linear combinations of partitions.
\end{rem}

Similarly as commutativity is represented by the partition $\crosspart$, the $(-1)$-commu\-ta\-tiv\-ity is represented by the linear combination $-\crosspart+2\Paaaa$. In particular the quantum group $O_N^{-1}$ is represented by the linear category of partitions $\langle\Paaaa-2\crosspart\rangle\nlin$ \cite{GWgen}. 

\begin{rem}
As a consequence, we have that $H_N=O_N\cap O_N^{-1}$. So, $O_N^{-1}$ can be seen not only as a~deformation of $O_N$, but also as a~liberation of $H_N$. 
\end{rem}

Actually, we have the following.

\begin{prop}[{\cite[Propositions 4.5, 6.9]{GWgen}}]
For any $N\in\N$, there is an isomorphism of monoidal $*$-categories $\phi\colon\langle\crosspart,\fourpart\rangle\nlin\to\langle\crosspart,\fourpart\rangle\nlin$ mapping
$$\crosspart\mapsto -\crosspart+2\Paaaa,\qquad \Paaaa\mapsto\Paaaa.$$
\end{prop}

In this sense, we can extend Definition \ref{D.NCversion} as follows.

\begin{defn}
\label{D.NCversion2}
Let $\Kat\subset\Part\nlin$ be a~linear category of partitions containing $-\crosspart+2\Paaaa$. A~category $\Kat^+\subset NC\nlin$ is called a~\emph{non-crossing version} of $\Kat$ if $\langle\Kat^+,-\crosspart+2\Paaaa\rangle\nlin=\Kat$.
\end{defn}

Note that the definition may not be compatible with Def.~\ref{D.NCversion}. That is, considering a~category~$\Kat$ that contains both $\crosspart$ and $-\crosspart+2\connecterpart$, then its non-crossing version according to Def.~\ref{D.NCversion} may not satisfy Def.~\ref{D.NCversion2} and vice versa.

Finally, we mention certain category isomorphism from \cite{GWintsp}. We denote $\tau_{(N)}:=\idpart-\frac{2}{N}\disconnecterpart\in\Part\nlin$. It holds that $\tau_{(N)}\cdot\tau_{(N)}=\idpart$ (hence also $T_{\tau_{(N)}}^2=I$). Consequently, the collection of maps $\T_{(N)}\colon\Part\nlin(k,l)\to\Part\nlin(k,l)$ defined as $p\mapsto \tau_{(N)}^{\otimes l}p\tau_{(N)}^{\otimes k}$ for $p\in\Part\nlin(k,l)$ is a~category isomorphism. Given a~category $\Kat\subset\Part\nlin$ corresponding to some quantum group $G$, its image under $\T_{(N)}$ corresponds to a~similar quantum group $G'=T_{\tau_{(N)}}GT_{\tau_{(N)}}$.

\subsection{Alternative maps associated to partitions and Möbius inversion}

In \cite{Maa18} alternative maps associated to partition were defined. Take $N\in\N$. For a~partition $p\in\Part(k,l)$, we define $\hat T_p\colon(\C^N)^{\otimes k}\to(\C^N)^{\otimes l}$ as
$$\hat T_p(e_1\otimes\cdots\otimes e_k)=\sum_{j_1,\dots,j_l=1}^N\hat\delta_p(\mathbf{i},\mathbf{j})(e_{j_1}\otimes\cdots\otimes e_{j_l}),$$
where $\hat\delta_p(\mathbf{i},\mathbf{j})$ equals to one if and only if the indices are equal precisely when the corresponding points are in the same block, otherwise it equals to zero. (Recall that the condition for $\delta_p(\mathbf{i},\mathbf{j})=1$ was: if the points are in the same block, the indices must coincide. In contrast, for the $\hat\delta_p$ we require equivalence: points are in the same block if and only if indices coincide.) Then the definition can be linearly extended to the whole space $\Part\nlin(k,l)$.

\begin{ex}
\label{ex.hatsing}
Take $p=\singleton^{\otimes l}$. Then $\delta_{\singleton^{\otimes l}}(\emptyset,\mathbf{j})=1$ for every tuple $(j_1,\dots,j_l)$. On the other hand $\hat\delta_{\singleton^{\otimes l}}(\emptyset,\mathbf{j})=1$ if and only if all the indices $j_1,\dots,j_l$ are mutually distinct. If we choose $l=N$, then $\hat\delta_{\singleton^{\otimes N}}(\emptyset,\mathbf{j})=1$ if and only if $(j_1,\dots,j_N)$ is a~permutation of the set $\{1,\dots,N\}$. If we choose $l>N$, then $\hat\delta_{\singleton^{\otimes l}}(\emptyset,\mathbf{j})=0$ for any~$\mathbf{j}$.
\end{ex}

It was shown in \cite[Lemma 4.24]{Maa18} that for any partition $p\in\Part(k,l)$, we have
$$T_p=\sum_{q\ge p}\hat T_q,$$
where we denote $q\ge p$ for $q\in\Part(k,l)$ if the partition $q$ was made from $p$ by merging some blocks.

It should be pointed out that the map $p\mapsto\hat T_p$ is not a~functor. Nevertheless, it still might be quite useful. The above mentioned formula shows us, that we can describe the linear map $\hat T_p$ by applying the functor $p\mapsto T_p$ on a~certain linear combination of partitions. More precisely, we can do the following.

For every partition $p\in\Part(k,l)$, we define a~linear combination $\hat p\in\Part\nlin(k,l)$ such that we have $p=\sum_{q\ge p}\hat p$. It can be easily seen that such an operation is indeed well defined. Actually, one can express explicitly the image $\hat p$ using the Möbius inversion formula as $\hat p=\sum_{q\ge p}\mu(p,q)q$, where $\mu$ is the Möbius function on the lattice of partitions, see \cite[Lecture 10]{NS06}. We extend this definition linearly to the whole space $\Part\nlin(k,l)$. Then we have $T_{\hat p}=\hat T_p$.

\begin{ex}
Let us compute $\widehat{\Labc}$ from the relation $p=\sum_{q\ge p}\hat q$. In order to do that, we actually have to compute $\hat p$ for every $p\in\Part(0,3)$. So, the largest partition with respect to the order $\le$ is the block partition $\Laaa$. Here we directly have $\Laaa=\widehat{\Laaa}$. Now, we have $\Laab=\widehat{\Laab}+\widehat{\Laaa}$, so $\widehat{\Laab}=\Laab-\Laaa$. Similarly, we can compute $\widehat{\Laba}$ and $\widehat{\Labb}$. Finally, we have
$$\Labc=\widehat{\Labc}+\widehat{\Laab}+\widehat{\Laba}+\widehat{\Labb}+\widehat{\Laaa}=\widehat{\Labc}+\Laab+\Laba+\Labb-2\Laaa,$$
so $\widehat{\Labc}=\Labc-\Laab-\Laba-\Labb+2\Laaa$.
\end{ex}

\subsection{Coxeter groups of type $D$}

The Coxeter group of type $D$ with $N$ generators, sometimes refered to as the \emph{demihyperoctahedral group}, is defined by the Coxeter diagram
$$\begin{tikzpicture}
\draw ( 0.3,0.7) circle (2pt) [fill]  --
      ( 1,0) circle (2pt) [fill]  --
      ( 2,0) circle (2pt) [fill]  --
      ( 3,0) circle (2pt) [fill];
\draw [dotted] (3,0) -- (4,0);
\draw ( 4,0) circle (2pt) [fill]  --
      ( 5,0) circle (2pt) [fill];
\draw ( 0.3,-0.7) circle (2pt) [fill]  --
      ( 1,0);
\end{tikzpicture}$$
Strictly speaking, Coxeter groups of type $D$ are defined only for $N\ge 4$ since for $N=1,2,3$ the definition coincides with the type $A$. Hence, Coxeter group with four generators corresponding to the diagram
\begin{tikzpicture}[x={(1ex,0ex)},y={(0ex,1ex)},baseline=-0.25em]

\draw ( 0.3,0.7) circle (1pt) [fill]  --
      ( 1,0) circle (1pt) [fill]  --
      ( 2,0) circle (1pt) [fill];
\draw ( 0.3,-0.7) circle (1pt) [fill]  --
      ( 1,0);
\end{tikzpicture}
can be defined as
$$\langle g_0,g_1,g_2,g_3\mid g_ig_j=g_jg_i,\;g_ig_0g_i=g_0g_ig_0\;\forall i,j=1,2,3\}.$$

Coxeter group of type $D$ with $N$ generators can be also realized as a~matrix group using matrices of size $N$ as follows
$$D_N:=\{X\in H_N\mid \perm X=1\},$$
where $H_N$ is the hyperoctahedral group represented by signed permutations and $\perm X$ is the {\em permanent} of a~matrix $X$ defined as
$$\perm X=\sum_{\sigma\in S_N}X_{1\sigma(1)}\cdots X_{N\sigma(N)}.$$
In this case, the permanent just multiplies the signs of the signed permutation.

\begin{rem}
The permanent as a~map $\perm\colon\GL(n,\C)\to\C$ is not a~homomorphism. Nevertheless it acts as a~homomorphism on the subgroup $H_N\subset\GL(n,\C)$ of signed permutation matrices. There is a~quantum group explanation for that. The signed permutation matrices actually have $(-1)$-commutative entries since there is just one non-zero element in each row and column. As we mentioned in Section \ref{secc.qcom}, permanent is also a~$(-1)$-determinant, which serves as a~representation of quantum groups with $(-1)$-commutative matrix entries.
\end{rem}

Our goal is to find a~free quantum analogue for the Coxeter groups of type $D$. Note that we already know one quantum group, whose classical version is $D_N$, namely the $(-1)$-deformation of the special orthogonal group $SO_N^{-1}$. Indeed, $SO_N^{-1}$ is obtained from $O_N^{-1}$ by the relation $\perm u=1$. Since the classical version of $O_N^{-1}$ is $H_N$, adding the permanent relation pushes us to $D_N$. Nevertheless, $SO_N^{-1}$ should not be considered as a~good candidate for {\em free} counterpart of $D_N$ -- its commutativity relations are deformed, not liberated.

\section{Free permanent and free $D_4^+$}
\label{sec.perm}

As indicated in Sect.~\ref{secc.qgex}, a~good approach to define a~free version of a~group is to study its representation category using partitions. In the following, we find the intertwiner corresponding to the $(-1)$-determinant alias permanent and describe it as a~linear combination of partitions. Then we define its free version.

\subsection{Intertwiner corresponding to the permanent}

%

Let us fix a~natural number $N\in\N$. We define the following vector in $(\C^N)^{\otimes N}$

$$P_{(N)}:=\sum_{\sigma\in S_N}e_{\sigma(1)}\otimes\cdots\otimes e_{\sigma(N)},$$
That is, in coordinates
$$[P_{(N)}]_{i_1,\dots,i_N}=\begin{cases}1 & \text{if $(i_1,\dots,i_N)$ is a~permutation,}\\ 0 & \text{otherwise.}\end{cases}$$

\begin{prop}
\label{P.P}
Let $G$ be a~compact matrix quantum group with $N\times N$ fundamental representation $u$ with $(-1)$-commutative entries. Then $\det_{-1}u$ is a~subrepresentation of $u^{\otimes N}$ corresponding to the one-dimensional invariant subspace spanned by the vector $P_{(N)}$.
\end{prop}
\begin{proof}
Let us write
$$[u^{\otimes N}P_{(N)}]_{i_1,\dots,i_N}=\sum_{j_1,\dots,j_N=1}^Nu_{i_1j_1}\cdots u_{i_Nj_N}[P_{(N)}]_{j_1,\dots,j_N}=\sum_{\sigma\in S_N}u_{i_1\sigma(1)}\cdots u_{i_N\sigma(N)}.$$
Now, using the deformed commutativity, we can see that actually
$$[u^{\otimes N}P_{(N)}]_{i_1,\dots,i_N}=\begin{cases}\det_{-1} u & \text{if $(i_1,\dots,i_N)$ is a~permutation},\\ 0 & \text{otherwise.}\end{cases}$$
Indeed, if $(i_1,\dots,i_N)$ is a~permutation, then we have exactly the defining formula for the $(-1)$-determinant. If $i_j=i_k$ for some $j,k$, we have that $u_{i_1\sigma(1)}\cdots u_{i_N\sigma(N)}=-u_{i_1(\sigma\circ\tau)(1)}\cdots u_{i_N(\sigma\circ\tau)(N)}$, where $\tau$ is the transposition of $j$ and $k$, so all the terms will cancel out.

So, we have $u^{\otimes N}P_{(N)}=\det_{-1}u\cdot P_{(N)}\in O(G)\otimes\spanlin\{P_{(N)}\}$, which is what we wanted.
\end{proof}

\begin{lem}
\label{L.P}
It holds that $P_{(N)}=\hat T_{\singleton^{\otimes N}}=T_{\widehat{\singleton^{\otimes N}}}$.
\end{lem}
\begin{proof}
Follows from Example \ref{ex.hatsing}
\end{proof}

In terms of intertwiners, Proposition \ref{P.P} tells us that if $u$ has $(-1)$-commutative entries, then we have $P_{(N)}P_{(N)}^*\in\Mor(u^{\otimes N},u^{\otimes N})$ (since $P_{(N)}P_{(N)}^*$ is the projection onto the subspace spanned by $P_{(N)}$). In addition, the equality $\det_{-1}u=1$ is equivalent to saying $P_{(N)}\in\Mor(1,u^{\otimes N})$. Lemma \ref{L.P} then allows us to formulate everything in terms of partitions since the vector $P_{(N)}$ can be represented by the partition $\widehat{\singleton^{\otimes N}}$.

\begin{rem}
Although $H_N$ is represented by a~matrix with $(-1)$-commutative entries, the category $\langle\crosspart,\fourpart\rangle\nlin$, which is usually used to describe this group, does not contain the element $\widehat{\singleton^{\otimes N}}\widehat{\singleton^{\otimes N}}^*$. This serves as an illustration of what was mentioned in Remark \ref{R.notinj}
\end{rem}

The above formulated proposition is is by no means a~new result. The intertwiner $P_{(N)}$ was used already in \cite{Wor88} to define the quantum group $SU^q_N$. The work of Woronowicz actually provides also kind of a~converse to Proposition \ref{P.P}

\begin{prop}
Let $G$ be a~compact matrix quantum group with $N\times N$ fundamental representation $u$ such that $P_{(N)}\in\Mor(1,u^{\otimes N})$. Then $u$ has $(-1)$-commutative entries.
\end{prop}
\begin{proof}
We follow the computation from \cite[page 66]{Wor88}. It is enough to show that $P_{(N)}\in\Cat_G$ implies $T_{-\crosspart+2\connecterpart}\in\Cat_G$. Let us compute the following
\begin{align*}
&[(\id\otimes\id\otimes P_{(N)}^*)(P_{(N)}\otimes\id\otimes\id)]_{(i_1i_2),(j_1j_2)}\\
&\qquad=\sum_{k_1,\dots,k_{N-2}}[P_{(N)}]_{i_1i_2k_1\cdots k_{N-2}}[P_{(N)}]_{k_1\cdots k_{N-2}j_1j_2}\\
&\qquad=\begin{cases}(N-2)!&i\neq j,\; \{i,j\}=\{k,l\}\\0&\text{otherwise}\end{cases}\\
&\qquad=(N-2)!\,(\delta_{\Pabba}(i_1i_2,j_1j_2)+\delta_{\crosspart}(i_1i_2,j_1j_2)-2\delta_{\connecterpart}(i_1i_2,j_1j_2)).
\end{align*}

Consequently,
\[\Cat_G\owns(1_N\otimes 1_N\otimes P_{(N)}^*)(P_{(N)}\otimes 1_N\otimes 1_N)=(N-2)!\,(T_{\Pabba}+T_{\crosspart}-2T_{\connecterpart})\qedhere\]
\end{proof}

\begin{ex}
We can define the following categories.
\begin{itemize}
\item The category $\langle\widehat{\singleton^{\otimes N}}\rangle\nlin=\langle-\crosspart+2\Paaaa,\widehat{\singleton^{\otimes N}}\rangle\nlin$ corresponds to the quantum group $SO_N^{-1}$.
\item The category $\langle\widehat{\singleton^{\otimes N}},\fourpart\rangle\nlin=\langle-\crosspart+2\Paaaa,\widehat{\singleton^{\otimes N}},\fourpart\rangle\nlin$ corresponds to the Coxeter group $D_4$.
\end{itemize}
\end{ex}

\subsection{Freeing the permanent relation, free $SO_4^{-1}$}

For this subsection, we set $N=4$. The key observation here is the fact that
\begin{equation}
\label{eq.abcdT}
\widehat{\Labcd}=-4\T_{(4)}\fourpart-(2\fourpart-\Labab)+\Laabb+\Labba.
\end{equation}

Hence, it might be a~good idea to substitute $\widehat{\Labcd}$ with $\T_{(4)}\fourpart$.

\begin{prop}
\label{P.SONC}
The category $\langle\T_{(4)}\fourpart\rangle\flin$ is a~non-crossing version of the category $\langle\widehat{\Labcd}\rangle\flin$.
\end{prop}
\begin{proof}
Obviously is $\langle\T_{(4)}\fourpart\rangle\flin$ noncrossing. From Eq.\ \eqref{eq.abcdT}, we see that
\begin{equation}
\label{eq.Tcross}
\langle\T_{(4)}\fourpart,-\crosspart+2\connecterpart\rangle\flin=\langle\widehat{\Labcd},-\Labab+2\Laaaa\rangle\flin=\langle\widehat{\Labcd}\rangle\flin.
\end{equation}
\end{proof}

\begin{defn}
We denote $S^{-1}O^+_4:=T_{\tau_{(4)}}^{\otimes 4}H_4^+T_{\tau_{(4)}}^{\otimes 4}$ the quantum group corresponding to the category $\langle\T_{(4)}\fourpart\rangle\flin$ and call it the \emph{free $(-1)$-special orthogonal group}.
\end{defn}

\begin{thm}
\label{T.SOfree}
We have
$$SO_4^{-1}\subsetneq S^{-1}O_4^+\subsetneq O_4^+.$$
The subgroup $SO_4^{-1}$ is obtained from $S^{-1}O_4^+$ imposing $(-1)$-commutativity on the entries of the fundamental representation.
\end{thm}
\begin{proof}
The statement follows from the partition picture. For the associated categories of partitions, we surely have
$$\langle\widehat{\Labcd},-\crosspart+2\connecterpart\rangle\flin\supsetneq\langle\T_{(4)}\fourpart\rangle\flin\supsetneq\langle\rangle\flin.$$
The right inclusion including the strictness is obvious. The left one follows from Eq.~\eqref{eq.Tcross}. It is strict since the second category is non-crossing while the first one is not. The inclusions for the categories imply corresponding inclusions for the quantum groups in the sense of Proposition~\ref{P.subgrp}. The inclusions indeed carry over including the strictness thanks to Lemma~\ref{L.strict} -- we work with $N=4$ and all the partitions are defined on $k+l=4$ points. The last assertion follows from Proposition~\ref{P.SONC}.
\end{proof}

\begin{rem}
Let us state two obvious remarks following directly from the definition of $S^{-1}O_4^+$. First, the quantum group is similar (and hence isomorphic) to the free hyperoctahedral group $H_4^+$. Secondly, the associated category $\Cat_{S^{-1}O_4^+}$ is isomorphic to the category $\Cat_{H_4^+}$.
\end{rem}

\subsection{Free version of Coxeter group $D_4$}
In this section, we are going to present the first main result of this article -- the definition of the free quantum analogue for the demihyperoctahedral group of rank four.

\begin{prop}
\label{P.DNC}
The category $\langle\T_{(4)}\fourpart,\fourpart\rangle\flin$ is a~non-crossing version of the category $\langle\widehat{\Labcd},\fourpart\rangle\flin$ in the sense of both Def.~\ref{D.NCversion} and Def~\ref{D.NCversion2}.
\end{prop}
\begin{proof}
Again, using Eq.~\eqref{eq.abcdT}, we can easily see that adding either the partition $\crosspart$ or the linear combination $-\crosspart+2\connecterpart$ to the category $\langle\T_{(4)}\fourpart,\fourpart\rangle\flin$ we obtain $\langle\widehat{\Labcd},\fourpart\rangle\flin$ in both cases.
\end{proof}

\begin{defn}
We denote $D_4^+:=H_4^+\cap S^{-1}O_4^+$ the quantum group corresponding to the category $\langle\T_{(4)}\fourpart,\fourpart\rangle\flin$ and call it the \emph{free demihyperoctehedral quantum group}.
\end{defn}

\begin{thm}
\label{T.D}
We have
$$D_4\subsetneq D_4^+\subsetneq H_4^+.$$
The subgroup $D_4$ is obtained from $D_4^+$ imposing commutativity or $(-1)$-commutativity on the entries of the fundamental representation.
\end{thm}
\begin{proof}
The proof is the same as in case of Theorem \ref{T.SOfree}. This time, we use the category inclusions
$$\langle\widehat{\Labcd},\fourpart\rangle\flin\supsetneq\langle\T_{(4)}\fourpart,\fourpart\rangle\flin\supsetneq\langle\fourpart\rangle\flin.$$
The right inclusion is obvious, its strictness follows from the fact that $\langle\fourpart\rangle$ contains only partitions with blocks of even length, whereas $\T_{(4)}\fourpart$ contains also odd blocks as summands. The left inclusion again follows from Eq.~\eqref{eq.Tcross}. It is strict since the second category is non-crossing while the first one is not. Again, those inclusions carry over to the corresponding quantum group including the strictness thanks to Proposition~\ref{P.subgrp} and Lemma~\ref{L.strict}. Finally, the last assertion follows from Proposition~\ref{P.DNC}.
\end{proof}

\subsection{Remarks on $N\neq 4$}

As a first remark, let us stress that our considerations are very specific for $N=4$. It is completely unclear, how to define the free analogues for other $N$. See also Question~\ref{Q.PcatN}.

Secondly, we can still ask whether the categories we just defined are interesting also for $N\neq 4$. (Although we should definitely not interpret them as free analogues of $D_N$ or $SO_N^{-1}$.) Unfortunately, we have the following.

\begin{lem}
Assume $N\neq 2,4$. Then
$$\langle\T_{(N)}\fourpart,\fourpart\rangle\nlin=\langle\fourpart,\singleton\otimes\singleton\rangle\nlin=\{\text{all NC partitions of even length}\}.$$
\end{lem}
\begin{proof}
It is straightforward to compute
\begin{align*}
&\UPartition{0:1,8}{0.2:2,7;0.4:3,6;0.6:4,5}\cdot(\T_{(N)}\fourpart\otimes\fourpart)\\
&\qquad=\left(1-\frac{6}{N}+\frac{12}{N^2}\right)\pairpart-\frac{2}{N}\left(1-\frac{2}{N}\right)\left(1-\frac{4}{N}\right)\singleton\otimes\singleton.
\end{align*}
It is known that $\langle\fourpart,\singleton\otimes\singleton\rangle\nlin$ is the category spanned by all non-crossing partitions of even length. In particular, $\T_{(N)}\fourpart\in\langle\fourpart,\singleton\otimes\singleton\rangle\nlin$.
\end{proof}

On the other hand, the category $\langle\T_{(N)}\fourpart\rangle\nlin$ makes perfect sense and, as we already mentioned, corresponds to the quantum group $T_{\tau_{(N)}}^{\otimes N}H_N^+T_{\tau_{(N)}}^{\otimes N}$, which is similar (and hence isomorphic) to $H_N^+$.

Finally, one might want to study the category $\Kat:=\langle\widehat{\singleton^{\otimes k}}\rangle\nlin$ for general $k$ and $N$. But we obtain nothing new except for the above discussed case $k=N$. For $k>N$, we have $\widehat{\singleton^{\otimes k}}=0$, so $\Kat=\langle\rangle\nlin$. For $k=1$, we have $\Kat=\langle\singleton\rangle$, for $k=2$, we have $\Kat=\langle\singleton\otimes\singleton\rangle\nlin$. One can show that for $2<k<N$, we have $\Kat=\langle\fourpart,\singleton\rangle\nlin=\{\text{all NC partitions}\}$ if $k$ is odd and $\Kat=\langle\fourpart,\singleton\otimes\singleton\rangle\nlin=\{\text{all NC partitions of even length}\}$ if $k$ is even.

\section{The representation category of $D_4^+$ as a~graph category}
\label{sec.graph}

Let us now informally summarize the goals and strategy for the next sections. Our aim is to study the representation category associated to the new free demihyperoctahedral quantum group $D_4^+$. Namely the most basic question reads:

\begin{quest}
Find an explicit description of the representation category $\Cat_{D_4^+}$.
\end{quest}

First, notice that we indeed do not have such a~description yet. Although we defined $D_4^+$ through its associated linear category of partitions, we know only the generators of the category, not all its elements explicitly. In contrast with the ``easy'' quantum groups, where we know, for instance, that $\langle\fourpart\rangle$ contains all non-crossing partitions with even blocks and hence we can explicitly write down how the associated representation category of $H_N^+$ looks like. In case of $D_4^+$, due to the fact that we are using linear combinations of partitions, it is not so straightforward to see, what are the explicit elements of our category.

A naive idea to describe all elements of the category $\Kat:=\langle\T_{(4)}\fourpart,\fourpart\rangle\flin$ is to notice that the category is linearly spanned by elements that were made from those two generators by finite amount of the category operations -- tensor product, composition, and involution. We can then model these operations by some graphs and these graphs would then correspond to elements of this category. To be more concrete, let us denote the four-block $\fourpart$ by a black point $\bcol$ (more precisely, by the diagram $\Gfour$) and the linear combination $\T_{(4)}\fourpart$ by a white point $\wcol$ (more precisely, by the diagram $\GWfour$). Let us now ignore the distinction between thin and thick lines for a moment. Now, we can construct more complicated diagrams that stand for some other elements of the category such as
$$
\Graph{
\GS 0.3/0.5:1/2.5,2/2.5,3/2.5;
\GS 0.3/0:1/,2/,3/;
\GV 0.5:2.5;
\GE 0.3/0.5:4/4.5,5/4.5,6/4.5;
\GE 0.3/0:4/,5/,6/;
\GW 0.5:4.5;
\GE 0.5/0.5:2.5/4.5;
}
=\UPartition{0:1,2,3,6,7,8}{0.6:4,5}\cdot(\fourpart\otimes\T_{(4)}\fourpart).
$$

In the end, every element of $\Kat$ can be expressed as a linear combination of such diagrams. Before discussing how much sensible and useful this approach is, let us note that there is also another way how to construct and interpret such diagrams with the same result: We forget about the distinction between black and white points -- both will now stand for the fourblock $\fourpart$ -- but we are going to distinguish between thick and thin lines -- thin lines will be simply identities, but the thick ones will stand for the partition $\tau_{(4)}$. 

This latter approach fits perfectly to the framework of \emph{graph categories} introduced recently in \cite{MR19}. We formalize this approach in Section~\ref{sec.graph}.

Nevertheless, one might object that such naive descriptions brings us basically nothing -- we are still not able to explicitly say what are the elements of the morphism spaces $\Cat_{D_4^+}(k,l)$ or $\Kat(k,l)$ since there are infinite amount of graphs/diagrams we would have to go through to determine this for every given $k,l$. In other words, there is an infinite amount of ways how to combine the generators together. Of course that if we iterate the category operations ad infinitum, we eventually obtain all elements of $\Kat(k,l)$ for some given $k,l$, but this is not practically possible. We have to find, for every $k,l\in\N_0$, a \emph{finite} amount of graphs (of manipulations with the generators) that already span the morphism space $\Kat(k,l)$. We solve this task in Section~\ref{sec.finite}.

\subsection{Graph categories}
\label{secc.graphcat}

As another example of monoidal involutive categories, we mention graph categories defined in \cite{MR19}. Actually, we generalize the approach of \cite{MR19} slightly.

\begin{defn}
A \emph{bilabelled graph} $\KG$ is a~triple $(K,{\bf a},{\bf b})$, where $K$ is a~graph, ${\bf a}=(a_1,\dots,a_k)$, ${\bf b}=(b_1,\dots,b_l)$ are tuples of vertices of $K$.
\end{defn}

The definition works for any type of graph. However, in this article, by a~graph we will always mean an undirected graph with the possibility of loops and multiple edges (unlike \cite{MR19}). We will call $\mathbf{a}$ the tuple of \emph{input vertices} while $\mathbf{b}$ are \emph{output vertices} (the role of $\mathbf{a}$ and $\mathbf{b}$ is switched in comparison with \cite{MR19} to be consistent with the notation for partitions). For any $k,l\in\N_0$ we denote by $\Gat(k,l)$ the set of all bilabelled graphs $\KG=(K,{\bf a},{\bf b})$ with $|{\bf a}|=k$, $|{\bf b}|=l$. The set of all bilabelled graphs is denoted simply by $\Gat$.

We define a~structure of a~(monoidal involutive) category on the set of all bilabelled graphs by introducing some operations. Consider $\KG=(K,{\bf a},{\bf b})$, $\HG=(H,{\bf c},{\bf d})$. We define

\begin{itemize}
\item \emph{tensor product} $\KG\otimes \HG=(K\sqcup H,{\bf ac},{\bf bd})$,
\item \emph{composition} (only defined if $|{\bf b}|=|{\bf c}|$) ${\bf H\cdot K}=(H\cdot K,{\bf a},{\bf d})$, where $H\cdot K$ is a~graph that is created from $H\sqcup K$ by contracting the vertices $b_i$ and $c_i$ for every $i$ (in contrast with \cite{MR19}, we keep the multiple edges),
\item \emph{involution} $\KG^*=(K,{\bf b},{\bf a})$.
\end{itemize}

We denote $\MG^{k,l}:=(M,v^k,v^l)$, where $M$ is a~graph with a~single vertex $v$. We denote by $\nullG$ the bilabelled graph corresponding to the \emph{null graph} with no vertices.

Any collection of bilabelled graphs $\Cat$ containing $\MG^{1,1}$ (playing the role of identity), $\MG^{0,2}$ (playing the role of the duality morphism), $\nullG$ (playing the role of the scalar identity) and closed under the above defined operations forms a~category called a~\emph{graph category}.

\begin{defn}
Let $A\in M_N(\R)$ be a~symmetric matrix, let $\KG=(K,{\bf a},{\bf b})$ be a~bilabelled graph and denote by $k$ and $l$ the length of ${\bf a}$ and ${\bf b}$. Then we define a~linear map $T^A_\KG\colon(\C^N)^{\otimes l}\to (\C^N)^{\otimes k}$ as
$$[T^A_\KG]_{\bf ij}:=\sum_{\substack{\phi\colon V(K)\to [N]\\\phi({\bf a})={\bf j},\;\phi({\bf b})={\bf i}}}\langle\phi\rangle_A,$$
where
$$\langle\phi\rangle_A:=\prod_{\{i,j\}\in E(K)}A_{\phi(i)\phi(j)}.$$
\end{defn}

\begin{prop}
The map $T^A$ is a~monoidal unitary functor. That is,
$$T^A_{\KG\otimes\HG}=T^A_\KG\otimes T^A_\HG,\quad T^A_{\HG\cdot\KG}=T^A_\HG T^A_\KG,\quad T^A_{\KG^*}=T^{A *}_\KG.$$
\end{prop}
\begin{proof}
The proof in all cases is straightforward. As an example, let us do the computation for the most complicated part, which is the composition.
\begin{align*}
[T^A_\HG T^A_\KG]_{\bf ij}&=\sum_{\bf k} [T^A_\HG]_{\bf ik}[T^A_{\KG}]_{\bf kj}=\sum_{\bf k}\sum_{\substack{\phi\colon V(H)\to[N]\\\phi({\bf c})={\bf k},\;\phi({\bf d})={\bf i}}}\sum_{\substack{\psi\colon V(K)\to[N]\\\psi({\bf a})={\bf j},\;\psi({\bf b})={\bf k}}}\langle\phi\rangle_A\langle\psi\rangle_A=\\
&=\sum_{\substack{\omega\colon V(H\sqcup K)\to[N]\\\omega({\bf a})={\bf j},\;\omega({\bf d})={\bf i}\\\omega({\bf b})=\omega({\bf c})}}\langle\omega|_{H}\rangle_A\langle\omega|_{K}\rangle_A=\sum_{\substack{\omega\colon V(H\cdot K)\to[N]\\\omega({\bf a})={\bf i},\;\omega({\bf d})={\bf j}}}\langle\omega\rangle_A=[T^A_{\HG\KG}]_{\bf ij}
\end{align*}
\end{proof}

\begin{rem}
The category of all partitions essentially embeds into the category of all graphs as $p\mapsto \KG_p:=(K_p,{\bf i},{\bf j})$, where $K_p$ is an edgeless graph with vertices corresponding to blocks of $p$, the tuples ${\bf i}$ and ${\bf j}$ describe the blocks of upper resp.\ lower points of $p$ \cite[Def.~6.13]{MR19}. See the examples below.

By \emph{essentially embeds}, we mean in particular that $T_p=T_{\KG_p}^A$ for all partitions $p$ and for any matrix $A$ with $\Tr A=N$. In order to make the mapping $p\mapsto\KG_p$ a true embedding, we would have to slightly modify either the definition of graph categories (by introducing a linear structure and imposing the relation $\MG^{0,0}=N\nullG$, cf.\ Sect.~\ref{sec.Ccat}; alternatively, by not considering the linear structure on partitions and allowing empty parts, see \cite[pp. 11--12]{MR19} or \cite{Ggrpgat}).
\end{rem}

We will denote bilabelled graphs pictorially. Unlike \cite{MR19}, we will draw them ``top to bottom'' instead of ``right to left''. Vertices will be denoted as black points and edges as thick lines. As in case of partitions, the whole graph should be drawn in a~strip between two horizontal lines. Input vertices are marked by drawing a~thin string from the top line to the corresponding vertex. Output vertices are connected by strings to the bottom line. The left-right order in which the strings are attached to the horizontal lines should correspond to the order of the tuples $\mathbf{a}$ and $\mathbf{b}$.

Let us mention a couple of examples. First, the edgeless graphs associated to partitions.

$$\MG^{1,1}=\KG_{\idpart}=\Gid,\quad \MG^{0,2}=\KG_{\pairpart}=\Gpair,$$
$$\MG^{2,2}=\KG_{\connecterpart}=\Gconnecter,\quad \MG^{0,4}=\KG_{\fourpart}=\Gfour,\quad \KG_{\crosspart}=\Gcross$$

Now, let us make some more complicated graphs and also illustrate how the operations work. Similarly to partitions, we can describe the operations as manipulations with the pictures. Tensor product is simply putting ``side by side'' while composition is putting ``one above the other'' and then contracting all the thin lines. Any bilabelled graph can be constructed using those operations starting with the edgeless partition graphs and the graph containing two points connected by an edge~$\GA$.
$$(\GA\otimes\GA)\cdot\Gconnecter\cdot(\GA\otimes\GA)=
\Graph{
\GE 1.25/1.75:1/,2/;
\GS 1/1.25:1/,2/;
\GS 1.75/2:1/,2/;
\GS 0.7/0.5:1/1.5,2/1.5;
\GS 0.3/0.5:1/1.5,2/1.5;
\GS 0.7/1:1/,2/;
\GS 0.3/0:1/,2/;
\GE -0.75/-0.25:1/,2/;
\GS -1/-0.75:1/,2/;
\GS -0.25/0:1/,2/;
\GV 1.25:1,2;
\GV 1.75:1,2;
\GV 0.5:1.5;
\GV -0.75:1,2;
\GV -0.25:1,2;
}
=\GEconnecter
$$
$$\Gmerge\cdot(\GA\otimes\GA)\cdot\Gfork=
\Graph{
\GS 1.3/1.5:1/1.5,2/1.5;
\GS 1.5/2:1.5/;
\GS 1.3/1:1/,2/;
\GE 0.25/0.75:1/,2/;
\GS 0/0.25:1/,2/;
\GS 0.75/1:1/,2/;
\GS -0.3/-0.5:1/1.5,2/1.5;
\GS -0.3/0:1/,2/;
\GS -0.5/-1:1.5/;
\GV 1.5:1.5;
\GV 0.25:1,2;
\GV 0.75:1,2;
\GV -0.5:1.5;
}
=\GAA
$$

In contrast with the thin strings, we draw the graph edges by thick lines.

As we already mentioned, we have $T^A_{\KG_p}=T_p$ for any partition $p$. In general, the idea of the $T^A_\KG$ map is similar as with partitions. The vertices correspond to Kronecker deltas, the thin strings correspond to identity maps and the thick edges correspond to the map $A$. In particular, we have
$$T^A_{\GA}=A,\qquad T^A_{\GAA}=A\bullet A,$$
where $\bullet$ denotes the entrywise {\em Schur product}, so $[A\bullet A]_{ij}=A_{ij}^2$.

\begin{rem}
In the original work \cite{MR19}, the matrix $A$ is considered to be an adjacency matrix of some graph $\Gamma$. In such a~case, it is indeed possible to ignore the multiplicity of the edges since we have $A_{ij}^2=A_{ij}$.
\end{rem}

\begin{rem}
Also here we can apply the Frobenius reciprocity. Categories of graphs are closed under rotations (and their inverses) defined as
\begin{align*}
\Lrot\KG&=(K,(a_2,\dots,a_k),(a_1,b_1,\dots,b_l)),\\
\Rrot\KG&=(K,(a_1,\dots,a_k,b_l),(b_1,\dots,b_{l-1}))
\end{align*}
for any $\KG=(K,\mathbf{a},\mathbf{b})\in\Gat(k,l)$. Hence, it is not that important which of the vertices are inputs and which are outputs. However, their cyclic order $a_k,\dots,a_1$, $b_1,\dots,b_l$ is important. This means that when drawing the bilabelled graphs, it is not that important to draw all the input string pointing up from the vertex and the output vertices pointing down. But it is important that all the strings point somehow outside the region, where the graph is drawn, so it is clear in which (say counter-clockwise) direction they appear.
\end{rem}

\begin{rem}
\label{R.fin}
Let us mention here the biggest disadvantage of describing a~representation category as a~graph category in comparison with categories of partitions. Let $\Cat$ be a~graph category. Then the morphism spaces $\Cat(k,l)$ are typically infinite for all $k,l\in\N_0$. So, even if we know all the elements of~$\Cat$ more or less explicitly, it might still be very hard to determine, how exactly the associated intertwiner spaces $\Mor(u^{\otimes k},u^{\otimes l})=\spanlin\{T^A_\KG\mid \KG\in\Cat(k,l)\}$ look like.

Therefore, given a~graph category $\Cat$ and a~symmetric matrix $A$ interpreting $\Cat$, we may be interested in finding a~collection of {\em finite} subsets $C(k,l)\subset\Cat(k,l)$ such that $T^AC(k,l)=T^A\Cat(k,l)$. We may even ask that $C(k,l)$ is the smallest possible collection. That is, to ask for $T^AC(k,l)$ being linearly independent.
\end{rem}

\subsection{Replacing edges by $\tau$}
Now our goal is to express the representation category $\Cat_{D_4^+}$ as an image of a~graph category by the functor $T^A$. To do that, we need to make a convenient choice for the matrix $A$. As we already mentioned, we basically need to be able to replace the thick edges by the element $\tau_{(4)}$.

\begin{prop}
Consider any $\pi\in\Part\nlin(1,1)$. Then there is a~monoidal unitary functor $\F_\pi\colon\Gat\to\Part\nlin$ mapping $\KG_p\mapsto p$ for every $p\in\Part$ and $\GA\mapsto\pi$. Moreover, we have $T^{T_\pi}_\KG=T_{\F_\pi\KG}$ for every $\KG\in\Gat$.
\end{prop}
\begin{proof}
Denote $\pi=\alpha\idpart+\beta\disconnecterpart$. Given a~bilabelled graph $\KG=(K,{\bf a},{\bf b})$, we denote by $p_\KG$ the partition of $({\bf a},{\bf b})$ according to the connected components of $K$. We denote by $\rl(\KG)$ the number of connected components $C$ of $K$ such that $a_i\not\in C$ and $b_j\not\in C$ for all $i,j$.

We can write explicitly the action of $\F_\pi$ as
$$\F_\pi\KG=\sum_{S\subset E(K)}\alpha^{|E(K)\setminus S|}\beta^{|S|}N^{\rl(\KG\setminus S)}p_{\KG\setminus S}.$$
Now we only need to check that this map indeed defines a~functor.
\end{proof}

The idea of the functor $\F_\pi$ is that given a bilabelled graph $\KG$, we replace every edge of $K$ by the partition $\pi$.

\begin{rem}
Note that the construction $\KG\mapsto p_\KG$ from the proof is a~special case of the functor $\F$ taking $\pi:=\idpart$. More precisely, $\F_{\idpart}p=N^{\rl(\KG)}p_{\KG}$.

Another canonical way of assigning partitions to bilabelled graphs is to map $\KG=(K,{\bf a},{\bf b})\mapsto \ker({\bf a},{\bf b})$ \cite[Def.~6.2]{MR19}. Here, $\ker({\bf a},{\bf b})$ denotes a~partition of $k$ upper and $l$ lowere points, where the upper points are labelled by the vertices $a_1,\dots,a_k$, the lower points are labelled by $b_1,\dots,b_l$, and two points are in the same part if the associated vertices coincide \cite[Sect.~2.4]{RW15}. In this case, we have $\F_{\disconnecterpart}\KG=N^{\#\text{inner vertices}}\ker({\bf a},{\bf b})$, where by \emph{inner vertices} we mean all vertices that are not input/output.
\end{rem}

Now, we can put $N=4$ and $\pi:=\tau:=\tau_{(4)}$ to obtain the following.

\begin{prop}
The quantum group $D_4^+$ is described by the graph category $\langle\Gconnecter,\GEconnecter\rangle$ through the functor $T^A$ with $A=T_\tau$.
\end{prop}
\begin{proof}
Using the preceding proposition. We see that
\[\F_\tau\langle\Gconnecter,\GEconnecter\rangle=\langle\connecterpart,\T_{(4)}\connecterpart\rangle\flin=\langle\fourpart,\T_{(4)}\fourpart\rangle\flin\qedhere\]
\end{proof}

\subsection{Explicit description of the associated graph category}

\begin{defn}[{\cite[Def.~5.1]{MR19}}]
Given a~bilabelled graph $\KG=(K,{\bf a},{\bf b})\in\Gat(k,l)$, we define the graph $K^\circ$ as the graph obtained from $K$ by adding the cycle $\alpha_k,\dots,\alpha_1,$ $\beta_1,\dots,\beta_l$ of new vertices and edges $a_i\alpha_i$, $b_j\beta_j$ for every, $i,j$. We refer to the cycle as the \emph{enveloping cycle}. We further define $K^\odot$ by adding an additional vertex adjacent to every vertex of the enveloping cycle.
\end{defn}

\begin{defn}[{\cite[Def. 5.3, 5.4]{MR19}}]
A bilabelled graph $\KG$ is called \emph{planar} if $K^\odot$ is a~planar graph. Equivalently, $K^\circ$ must be planar and the enveloping cycle must be facial. 
\end{defn}

Basically, this means that we can draw the bilabelled graph $\KG$ according to the rules described in Sect.\ \ref{secc.graphcat} in such a~way that all the lines (both thick and thin, that is, edges of $K$ together with the input/output strings) do not cross each other. The face of $K^\circ$ corresponding to the enveloping cycle or the face of $K$ neighbouring with the input/output vertices is then drawn as the unbounded face and we will also refer to this face as the \emph{unbounded face} in the subsequent text.

\begin{rem}[{\cite[Thm.~5.15]{MR19}}]
The planar bilabelled graphs are closed under the category operations. Hence, they form a~graph category.
\end{rem}

\begin{defn}
\label{D.gendeg}
Let $\KG=(K,{\bf a},{\bf b})$ be a~bilabelled graph and let $v$ be a~vertex of~$K$. The degree of $v$ in $K^\circ$ will be called the \emph{extended degree} of $v$ and denoted by~$\tilde d_v$.
\end{defn}

\begin{defn}
Let $\KG$ be a~bilabelled graph containing an inner vertex $v$ (i.e.\ $v$ is not an input/output vertex) of degree two. Let $e,f$ be the two edges incident with $v$ and let $\KG'$ be a~bilabelled graph that was made from $\KG$ by contracting those two edges. We will call $\KG'$ a~\emph{two-path contraction} of $\KG$.
\end{defn}

\begin{defn}
For any set $S$ of bilabelled graphs, we denote by $\langle S\rangle^{\rm c}$ the smallest graph category containing $S$ that is closed also under two-path contraction.
\end{defn}

\begin{prop}
\label{P.iv}
Consider a~bilablled graph $\KG$ and its two-path contraction $\KG'$.
\begin{enumerate}
\item If $A^2=1$, then $T^A_{\KG}=T^A_{\KG'}$.
\item If $\pi\cdot\pi=\idpart$, then $\F_\pi\KG=\F_\pi\KG'$.
\end{enumerate}
In particular, the image of a~given graph category under such a functor $T^A$, resp.\ $\F_\pi$ coincides with the image of its two-path-contraction closure.
\end{prop}
\begin{proof}
The idea is the same in both cases: The relation $A^2=1$, resp.\ $\pi\cdot\pi=\idpart$ means that two consecutive edges may be replaced by a~single (thin) string and hence contracted without changing the image.
\end{proof}

So, instead of studying the category $\langle\Gconnecter,\GEconnecter\rangle$, we can study its two-path-contrac\-tion closure $\langle\Gconnecter,\GEconnecter\rangle^{\rm c}$.

\begin{defn}
We define ${\bf X}^{k,l}:=(X_{k+l},(v_1,\dots,v_k),(v_{k+1},\dots,v_{k+l}))$, where $X_n$ denotes the star graph with central vertex $v_0$ and surrounding vertices $v_1,\dots,v_n$.
\end{defn}

\begin{prop}
\label{P.cat}
We have
$$\langle\Gconnecter,\GEconnecter\rangle^{\rm c}=\langle\MG^{k,l},\XG^{k,l}\mid k+l\text{ even}\rangle=\Cat,$$
where $\Cat$ consists of all bilabelled graphs $\KG=(K,{\bf a},{\bf b})\in\Gat(k,l)$ such that
\begin{enumerate}
\renewcommand{\theenumi}{\roman{enumi}}
\item $\KG$ is planar,
\item $\tilde d_v$ is even for all $v\in V(K)$,
\item $K$ is bipartite and all the vertices $a_1,\dots,a_k,b_1,\dots,b_l$ are elements of one of the parts. (We will call this part \emph{even} and the other will be \emph{odd}.)
\end{enumerate}
\end{prop}

Note in particular that $\Cat$ does not contain graphs with loops since these would not be bipartite. We are going to prove this proposition by showing a~series of inclusions.

\begin{lem}
We have $\Cat\supset\langle\Gconnecter,\GEconnecter\rangle^{\rm c}$.
\end{lem}
\begin{proof}
Obviously we have $\Gconnecter,\GEconnecter\in\Cat$. Now it is easy to check that $\Cat$ is closed under the category operations and two-path contractions.
\end{proof}

\begin{lem}
We have $\langle\Gconnecter,\GEconnecter\rangle^{\rm c}\supset\langle\MG^{k,l},\XG^{k,l}\mid k+l\text{ even}\rangle$.
\end{lem}
\begin{proof}
The fact that all $\MG^{k,l}$ for $k+l$ even are generated by $\Gconnecter$ is basically known from the theory of partition categories (see \cite[Prop.~2.7(2)]{Web13}). To prove it, it is enough to show that $\MG^{0,l}\in\langle\Gconnecter\rangle\subset\langle\Gconnecter,\GEconnecter\rangle^{\rm c}$ for every $l$ even. For $l=2,4$ we have it by definition. For $l=0$, we have $\MG^{0,0}=\MG^{2,0}\MG^{0,2}$. For larger $l$ we prove it by induction. We have $\MG^{0,2l+2}=(\MG^{1,1\,\otimes (l-1)}\otimes\MG^{2,0}\otimes\MG^{1,1\,\otimes 3})(\MG^{0,2l}\otimes\MG^{0,4})$.

The proof for $\XG^{k,l}\in\langle\GEconnecter\rangle^{\rm c}\subset\langle\Gconnecter,\GEconnecter\rangle^{\rm c}$ goes the same way except for the fact that we have to use a~two-path contraction at the end.
\end{proof}

\begin{lem}
We have $\langle\MG^{k,l},\XG^{k,l}\mid k+l\text{ even}\rangle\supset\Cat$.
\end{lem}
\begin{proof}
The idea of the proof is that given any $\KG\in\Cat$, we take $\MG^{k+l}$ for every even vertex of $K$ and $\XG^{k+l}$ for every odd vertex of $K$ and compose them together as prescribed by the graph $K$. We illustrate the idea on the following example, where we for clarity denote the even vertices by black circles~$\bcol$ and the odd vertices by white circles~$\wcol$.
\begin{align*}\displaystyle
\Graph{
\GS 3.5/4:1.5/,2.5/,3/,3.5/,4/,4.5/;
\GS 3/3.5:2/1.5,2/2.5,3/,4/3.5,4/4,4/4.5;
\GE 2/3:1/2,3/2,3/,3/4;
\GE 1/2:2/1,2/3;
\GE 0/1:1.85/,2.15/;
\GE -1/0:1/2,3/2;
\GS -1.5/-1:0/1,0.5/1,1/,1.5/1,2/1,3/;
\GS -2/-1.5:0/,0.5/,1/,1.5/,2/,3/;
\GV 3:2,3,4;
\GW 2:1,3;
\GV 1:2;
\GW 0:2;
\GV -1:1,3;
}
&=
\Graph{
\GS 3.5/4:1.5/,2.5/,3/,3.5/,4/,4.5/;
\GS 3/3.5:2/1.5,2/2.5,3/,4/3.5,4/4,4/4.5;
\GS 2.5/3:1.5/2,2.5/2,3/,3.5/4;
\GE 2/2.5:1/1.5,3/2.5,3/,3/3.5;
\GE 1.5/2:1.5/1,2.5/3;
\GS 1/1.5:2/1.5,2/2.5;
\GS 0.5/1:1.5/2,2.5/2;
\GE 0/0.5:2/1.5,2/2.5;
\GE -0.5/0:1.5/2,2.5/2;
\GS -1/-0.5:1/1.5,3/2.5;
\GS -1.5/-1:0/1,0.5/1,1/,1.5/1,2/1,3/;
\GS -2/-1.5:0/,0.5/,1/,1.5/,2/,3/;
\GV 3:2,3,4;
\GV 2.5:1.5,2.5,3,3.5;
\GW 2:1,3;
\GV 1.5:1.5,2.5;
\GV 1:2;
\GV 0.5:1.5,2.5;
\GW 0:2;
\GV -0.5:1.5,2.5;
\GV -1:1,3;
}\\&=
(\MG^{1,5}\otimes\MG^{1,1})\XG^{2,2}\MG^{2,2}(\XG^{1,1}\otimes\XG^{3,1})(\MG^{2,2}\otimes\MG^{1,1}\otimes\MG^{3,1}).
\end{align*}

A proper proof should be formulated using the induction. This requires some preparation and we devote to it a~special subsection \ref{secc.induction}.
\end{proof}

\begin{proof}[Proof of Proposition \ref{P.cat}]
Follows from the lemmata above.
\end{proof}

\subsection{Inductive description of $\Cat$}
\label{secc.induction}

Given a~bilabelled graph $\KG=(K,\mathbf{a},\mathbf{b})\in\Cat$, we define $K^\bullet$ to be the subgraph of $K$ obtained by removing all vertices $v$ with $d_v=1$, $\tilde d_v=2$ in $K$. We also define $\mathbf{a}^\bullet$ and $\mathbf{b}^\bullet$ as follows: if $a_i\in K^\bullet$, then $a_i^\bullet:=a_i$; otherwise, $a_i^\bullet$ is the unique neighbor of $a_i$. The same definition is used for $\mathbf{b}^\bullet$. For any vertex $v$ of $K^\bullet$ we will denote $d_v^\bullet$ its degree inside $K^\bullet$ (while $d_v$ denotes its degree inside $K$). Note that when representing a~bilabelled graph by a~picture, it is perfectly enough to draw $K^\bullet$ instead of $K$ if we keep the distinction between thin and thick lines or, equivalently, between even and odd vertices. Here follows an example.

$$
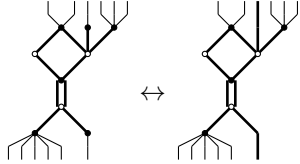
\Graph{
\GS 3.5/4:1.5/,2.5/,3/,3.5/,4/,4.5/;
\GS 3/3.5:2/1.5,2/2.5,3/,4/3.5,4/4,4/4.5;
\GE 2/3:1/2,3/2,3/,3/4;
\GE 1/2:2/1,2/3;
\GE 0/1:1.85/,2.15/;
\GE -1/0:1/2,3/2;
\GS -1.5/-1:0/1,0.5/1,1/,1.5/1,2/1,3/;
\GS -2/-1.5:0/,0.5/,1/,1.5/,2/,3/;
\GV 3:2,3,4;
\GW 2:1,3;
\GV 1:2;
\GW 0:2;
\GV -1:1,3;
}
\leftrightarrow
\Graph{
\GS 3.5/4:1.5/,2.5/,3.5/,4/,4.5/;
\GS 3/3.5:2/1.5,2/2.5,4/3.5,4/4,4/4.5;
\GE 3/4:3/;
\GE 2/3:1/2,3/2,3/,3/4;
\GE 1/2:2/1,2/3;
\GE 0/1:1.85/,2.15/;
\GE -1/0:1/2,3/2;
\GE -2/-1:3/;
\GS -1.5/-1:0/1,0.5/1,1/,1.5/1,2/1;
\GS -2/-1.5:0/,0.5/,1/,1.5/,2/;
\GV 3:2,4;
\GW 2:1,3;
\GV 1:2;
\GW 0:2;
\GV -1:1;
}
$$


We say that an element $v$ \emph{appears consecutively} in the sequence $c_1,\dots,c_m$ if its occurence forms an interval up to rotation. That is, there is an index $i\in\{1,\dots,m\}$ and a~number $r\in\{0,\dots,m-1\}$ such that $c_{i+j}=v$ if and only if $j\in\{0,\dots,r\}$, where the indices are taken modulo $m$.

Now, we present the inductive description of $\Cat$, where the induction is on the number of vertices of $K^\bullet$. So, consider $\KG=(K,{\bf a},{\bf b})\in\Cat(k,l)$. Take any vertex $v$ bordering with the unbounded face in $K^\bullet$ that appears consecutively among $a_k^\bullet,\dots,a_1^\bullet,b_1^\bullet,\dots,b_l^\bullet$ (or does not appear there at all). For simplicity, we can assume that $\mathbf{b}^\bullet=(v,v,\dots,v)$ and that $\mathbf{a}^\bullet$ does not contain $v$ (otherwise rotate $\KG$). In particular $l=\tilde d_v-d_v^\bullet$. Denote also $m:=d_v^\bullet$ and $c_1,\dots,c_m$ the neighbors of $v$ in $K^\bullet$.

We distinguish two cases -- (a) $v$ is even or (b) $v$ is odd. It holds that
\begin{enumerate}
\renewcommand{\theenumi}{\alph{enumi}}
\item $\KG={\bf M}^{m,l}\HG$, resp.
\item $\KG={\bf X}^{m,l}\HG$,
\end{enumerate}
where, in both cases, $\HG$ was made from $\KG$ by removing the vertex $v$ from $K^\bullet$. To be more precise, we have (a) $\HG=(H,{\bf a}',{\bf d})$, where $H$ was made from $K$ by removing the vertex $v$, adding vertices $d_1,\dots,d_m$ and edges $\{c_i,d_i\}$, and (b) $\HG=(H,{\bf a}',{\bf c})$, where $H$ was made from $K$ by removing the vertices $v,b_1,\dots,b_l$.

Let us formalize this procedure to an algorithm.

\begin{alg}
\label{A.naive}
Constructing the sets $\Cat(k,l)$, $k,l\in\N_0$.
\begin{enumerate}
\item Add the empty graph to $\Cat(0,0)$.
\item For every $k,l,m\in\N_0$ such that $k+m$ and $l+m$ are even and for every $\HG=(H,{\bf a},{\bf b})\in\Cat(k,m)$, add (b) $\XG^{m,l}\HG$ to $\Cat(k,l)$ and (a), if all $b_1^\bullet,\dots,b_m^\bullet$ are odd, add also $\MG^{m,l}\HG$ to $\Cat(k,l)$.
\item Add also all rotations of the graphs constructed in (2) into $\Cat$.
\item Repeat (2) and (3).
\end{enumerate}
\end{alg}

\begin{rem}
The algorithm itself is very banal. Basically it says nothing more than ``graphs are made from smaller graphs by adding vertices''. Nevertheless, we are going to improve this algorithm in the following section and hence it will be useful to have this naive version formulated as well.

Actually, it is not even an algorithm since it obviously never terminates. To make it terminate, we would have to introduce some condition such as restricting the number of vertices of $K^\bullet$. In this way, we would end up with a~collection of some subsets $C(k,l)\subset\Cat(k,l)$. As we already mentioned in Remark \ref{R.fin}, it is in general not clear how to formulate such a~terminating condition in order to obtain sufficiently large sets $C(k,l)$ in order to describe the whole image by the associated functor $T^A$. We attack this problem in the following section.
\end{rem}

\section{Finite generating set for the intertwiner spaces}
\label{sec.finite}

Denote $\Cat:=\langle\Gconnecter,\GEconnecter\rangle^{\rm c}$ as in Proposition \ref{P.cat}. In this section, we are going to attack the problem indicated in Remark \ref{R.fin}. We are going to define a~collection of finite subsets $C(k,l)\subset\Cat(k,l)$ such that $\spanlin T^{T_\tau}_{C(k,l)}=T^{T_\tau}_{\Cat(k,l)}=\Cat_{D_4^+}(k,l)$. Only this allows us to explicitly compute the intertwiner spaces $\Cat_{D_4^+}(k,l)$.

\subsection{Defining the subsets $C(k,l)$}

Denote $C(k,l)\subset\Cat(k,l)$ a~subclass of bilabelled graphs $\KG=(K,{\bf a},{\bf b})$ that satisfies in addition the following
\begin{enumerate}\setcounter{enumi}{3}
\renewcommand{\theenumi}{\roman{enumi}}
\item there is no vertex $v$ with $d_v=\tilde d_v=2$,
\item there are no multiple edges,
\item every component of $\KG$ has at least one vertex among $(\mathbf{a},\mathbf{b})$.
\end{enumerate}

The condition (iv) can be equivalently formulated as $\tilde d_v\ge 4$ for every vertex $v$ of $K^\bullet$ unless $v$ is isolated. The properties are not preserved under the category operations, so $C$ is not a~graph category. Nevertheless, the following holds.

\begin{prop}
\label{P.FCCat}
We have $\F_\tau C(k,l)=\F_\tau\Cat(k,l)$. Hence, $T^{T_\tau}_{C(k,l)}=T^{T_\tau}_{\Cat(k,l)}$.
\end{prop}
\begin{proof}
It is enough to show that $\F_\tau\KG$ is invariant (up to scalar factors) under (iv)~two-path contractions, (v)~erasing pairs of edges between two vertices in $\KG$, (vi)~erasing components of $K$ with no input/output vertex.

Invariance of $\F_\tau\Cat(k,l)$ under (iv) was shown already in Proposition \ref{P.iv}.

For (v), let $\KG$ be a~bilabelled graph containing two vertices $v,w$ that are connected by multiple edges. Let $\KG'$ be made from $\KG$ by erasing two of those edges. Then it holds that $\F_\tau\KG=1/4\,\F_\tau\KG'$. This follows from the fact that $\F_\tau\GAA=1/4\,\disconnecterpart$.

Finally, if $\KG$ has a~component with no input/output vertex, then it is of the form $\KG=\KG'\otimes\HG$, where $H$ has no input output vertex and hence $\F_\tau\HG=\alpha\in\C$. But this means that $\F_\tau\KG=\alpha\F_\tau\KG'$.
\end{proof}

\begin{lem}
\label{L.compinC}
Consider $\KG\in C(k,l)$. Then every bilabelled graph $\HG$ induced by a~connected component $H$ of $K$ is in $C$.
\end{lem}
\begin{proof}
Straightforward observation.
\end{proof}

\begin{thm}
\label{T.Cfin}
For every $k,l\in\N_0$, $C(k,l)$ is a~finite set.
\end{thm}

We prove this theorem in Section \ref{secc.proof}.

\subsection{Algorithm for generating $C(k,l)$}

In this subsection, we give an algorithm to compute explicitly the sets $C(k,l)$ for any $k,l\in\N_0$. In the following subsection, we are going to prove Theorem \ref{T.Cfin} by showing that this algorithm always terminates.

\begin{lem}
\label{L.induction}
Consider $\KG\in C(k,l)$ such that $K$ is connected and $|V(K^\bullet)|\ge 2$. Then there is a~vertex $v$ of $K^\bullet$ with $d^\bullet_v\le 2$ such that $K^\bullet\setminus\{v\}$ is connected.
\end{lem}
\begin{proof}
Suppose the contrary. That is, there are $m$ vertices $v_1,\dots,v_m$ in $K^\bullet$ such that $d_{v_i}^\bullet=2$ and $K\setminus\{v_i\}$ is not connected. Otherwise all vertices in $K^\bullet$ have degree at least three. Denote $n:=|V(K^\bullet)|$ the number of vertices in $K^\bullet$. Denote by $c$ the number of vertices bordering with the unbounded face. From (ii), (iv) it follows that those are the only vertices that can have degree two or three, all other vertices have degree at least four. Consequently, we can make the following estimate
\begin{equation}\label{eq.ineq}
|E(K^\bullet)|={1\over 2}\sum_{v\in V(K^\bullet)}d_v^\bullet\ge 2n-{1\over 2}c-{1\over 2}m.
\end{equation}

Now, recall also the Euler's formula for non-empty connected planar graphs
$$\#\hbox{faces}+\#\hbox{vertices}-\#\hbox{edges}=2.$$
From this formula, we can infer that the number of edges surrounding the unbounded component is at least $c-1$. Denote by $\Phi$ the set of all faces in $K^\bullet$. For a~face $\Omega\in\Phi$, denote by $\nu(\Omega)$ the number of edges bordering $\Omega$, where we count an edge twice if it borders with $\Omega$ from both sides. In particular $\nu(\Omega_{\infty})\ge c-1+m$, where $\Omega_\infty$ is the unbounded face. Since the graph is bipartite and has no multiple edges, we have $\nu(\Omega)\ge 4$ for all the other faces $\Omega\in\Phi$. Now, we can estimate
$$2|E(K^\bullet)|=\sum_{\Omega\in\Phi}\nu(\Omega)\ge4(|\Phi|-1)+c-1+m\ge4(1+|E(K^\bullet)|-n)+c-1+m,$$
so
\begin{equation}
|E(K^\bullet)|\le 2n-{3\over 2}-{1\over 2}c-{1\over 2}m.
\end{equation}
This is, however, in contradiction with Inequality~\eqref{eq.ineq}.
\end{proof}

\begin{rem}
Given $\KG$ satisfying conditions (i)--(v), then the condition (vi) is equivalent to saying that there is no isolated vertex in $K^\circ$. Indeed, suppose $\KG$ satisfies (i)--(v), but not (vi). Then the same must hold also for the bilabelled graph $\HG\in\Cat(0,0)$ induced by the connected component with no input/output vertex. Since $\HG$ has no input/output vertices, we have $H^\bullet=H$. In the proof of Lemma \ref{L.induction} we actually did not use the assumption (vi), so there is a~vertex $v$ of $H=H^\bullet$ with $d_v=\tilde d_v=d_v^\bullet\le 2$. From (ii), (iv), we actually have $d_v=\tilde d_v=0$.
\end{rem}

\begin{lem}
\label{L.consec}
Consider $\KG=(K,{\bf a},{\bf b})\in C(k,l)$ such that $K$ is connected. Take any $v\in V(K^\bullet)$ such that $K\setminus\{v\}$ is connected. Then $v$ appears consecutively in $a_k^\bullet,\dots,a_1^\bullet,b_1^\bullet,\dots,b_l^\bullet$.
\end{lem}
\begin{proof}
Supposing $v$ does not appear consecutively, and assuming that $K\setminus\{v\}$ is connected, we can find $K_5$ as a~minor of $K^\odot$. Indeed, $v$ not appearing consecutively means that there exist vertices $w_1,w_2\in V(K)^\bullet$ such that the sequence $a_k^\bullet,\dots,a_1^\bullet$, $b_1^\bullet,\dots,b_l^\bullet$ contains $\dots w_1\dots v\dots w_2\dots v\dots$ Denote by $\alpha$, $\beta$, $\gamma$, and $\delta$ the corresponding vertices on the enveloping cycle. Since $K\setminus\{v\}$ is connected, there must be a~path between $w_1$ and $w_2$ not containing $v$. So, we can construct $K_5$ as a~minor as follows:
$$
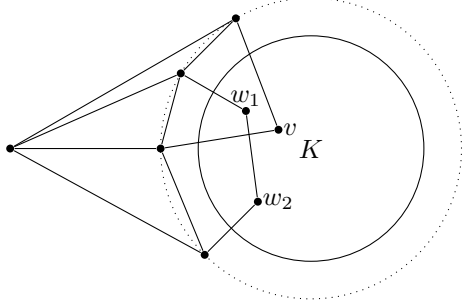
\begin{tikzpicture}[inner sep=0pt,minimum size=3pt]
\node at (0,0) [circle,draw,minimum size=3cm] {$K$};
\node at (canvas polar cs:angle=150,radius=0.5cm) [circle,fill,label=right:$v$] (v) {};
\node at (canvas polar cs:angle=150,radius=1cm)   [circle,fill,label=above:$w_1$] (w1) {};
\node at (canvas polar cs:angle=225,radius=1cm)   [circle,fill,label=right:$w_2$] (w2) {};
\node at (canvas polar cs:angle=120,radius=2cm)   [circle,fill] (a) {};
\node at (canvas polar cs:angle=150,radius=2cm)   [circle,fill] (b) {};
\node at (canvas polar cs:angle=180,radius=2cm)   [circle,fill] (c) {};
\node at (canvas polar cs:angle=225,radius=2cm)   [circle,fill] (d) {};
\node at (canvas polar cs:angle=180,radius=4cm)   [circle,fill] (inf) {};
\draw[dotted] circle (2cm);
\draw (inf)--(a);
\draw (inf)--(b);
\draw (inf)--(c);
\draw (inf)--(d);
\draw (a)--(v)--(c);
\draw (b)--(w1)--(w2)--(d);
\draw (a)--(b)--(c)--(d);
\end{tikzpicture}
$$
\end{proof}


\begin{alg}
\label{A}
Input: $k_0\in\N_0$. Output: The sets $C(k,l)$ for all $k,l$, $k+l\le k_0$.
\begin{enumerate}
\item[(1)] Add $\nullG$ into $C(0,0)$. For all $k,l\in\N_0$, $4\le k+l\le k_0$ add $\MG^{k,l}$ and $\XG^{k,l}$ into $C(k,l)$.
\item[(2A)] For every $k,l\in\N_0$ odd, $k+l\le k_0$, $l\ge 3$ and every $\HG=(H,{\bf a},b)\in C(k,1)$ add either ${\bf M}^{1,l}\HG$ (if $b\neq b^\bullet$ i.e.\ $b^\bullet$ is odd) or ${\bf X}^{1,l}\HG$ (if $b=b^\bullet$ is even)  to $C(k,l)$.
\item[(2B)] For every $k,l\in\N_0$ even, $k+l\le k_0$, $l\ge 2$ and every $\HG=(H,{\bf a},{\bf b})\in C(k,2)$ such that $b_1^\bullet\neq b_2^\bullet$ and both are either (a) odd or (b) even add either (a) ${\bf M}^{2,l}\HG$ or (b) ${\bf X}^{2,l}\HG$ to $C(k,l)$.
\item[(3)] Add into $C$ also all rotations of the bilabelled graphs obtained in (2A) and (2B).
\item[(4)] Repeat (2A)--(2B)--(3) until no new graphs appear.
\item[(5)] Add $\MG^{0,2}$, $\MG^{1,1}$, $\MG^{2,0}$ into $C$. Add all possible tensor products and their rotations to obtain bilabelled graphs $\KG$ with unconnected $K$.
\end{enumerate}
\end{alg}

\begin{prop}
Suppose that Algorithm \ref{A} terminates after a~finite number of steps. Then it indeed constructs all the elements of $C(k,l)$, $k+l\le k_0$, that is, bilabelled graphs satisfying (i)--(vi). In particular, this means that the sets $C(k,l)$ are finite.
\end{prop}
\begin{proof}
The algorithm is a~modification of Algorithm \ref{A.naive}. We just need to justify the changes.

First change is that in the steps (1)--(4) we construct only graphs $\KG\in C(k,l)$ such that $K$ is connected. This is possible thanks to Lemma \ref{L.induction} that says that to construct $\KG$ with $K$ connected, we only need to add a~vertex to a~graph $\HG$ such that $H$ is also connected. So, to obtain all connected graphs $\KG\in C(k,l)$, we do not have to consider the unconnected ones. Finally, as we mentioned in Lemma~\ref{L.compinC}, taking a~bilabelled graph $\KG$ satisfying (i)--(vi), its connected components also satisfy (i)--(vi). So, supposing the steps (1)--(4) terminate in finite time, we can obtain all unconnected elements of $C(k,l)$ just using the tensor product on the connected elements.

Second change is that, in the new algorithm, we perform the step~(2) only for $m=1$ (step~(2A)) and for $m=2$ (step~(2B)). First, notice that we do not have to consider the case $m=0$ since this would add a~new connected component to the graph. This would result in obtaining an unconnected graph unless the original one $\HG$ contains no vertex. This forces us, however, to change the starting point of the induction~(1), where it is not enough to consider just the empty graph, but we have to consider all the possible graphs with one vertex -- those are $\MG^{k,l}$ and $\XG^{k,l}$ (while we have to exclude the rotations of $\XG^{1,1}$ as they do not satisfy~(iv); the rotations of $\MG^{1,1}$ can be included at the end, because $\MG^{1,1}$ acts as identity with respect to composition). The fact that we can keep $m\le 2$ follows from Lemma \ref{L.induction}.

Now there are some changes related to the fact that we are not aiming to construct the whole sets $\Cat(k,l)$ of graph satisfying (i)--(iii), but only the subsets $C(k,l)$ satisfying in addition the conditions (iv)--(vi). In step~(2A), we cannot consider $l=1$ as $\XG^{1,1}\HG$ would violate~(iv) (and $\MG^{1,1}\HG=\HG$, so this also need not be considered), so $l\ge3$ here. In both steps (2A) and~(2B), we require $b_i^\bullet$ to be all odd when adding $\MG^{m,l}\HG$ since otherwise we would violate~(iv). In step (2B) we require $b_1^\bullet\neq b_2^\bullet$ since otherwise we would obtain a~double edge and hence violate~(v). We do not consider $l=0$ since this would violate (iv), so $l\ge 2$ here.

The condition~(vi) is satisfied automatically~-- every graph constructed in (2A)--(2B)--(3) is connected and has at least two input/output vertices. Consequently, every connected component of every graph constructed in~(5) has at least two input/output vertices.

Now the most important feature following from those changes is that the number of the input/output vertices can never decrease. It increases by $l-m$, which is non-negative since $m\le 2$, $l\ge2$. Actually the only way it can stay the same is when we choose $l=2$ in step~(2B), otherwise it strictly increases. This means that if we want to compute all $C(k,l)$ for $k+l\le k_0$, we can restrict to those spaces with $k+l\le k_0$ already in the beginning, which gives us hope to get the result in a~finite amount of steps. Restricting $m\le 2$ also means that in each step we add to $C(k,l)$ only a~finite amount of new elements. Consequently, if the algorithm finishes after a~finite amount of steps, it means that the sets $C(k,l)$ are finite.
\end{proof}

\begin{ex}
As an example, let us go through the algorithm for $k_0=8$. We will not distinguish between input and output vertices, so everything will be here ``up to rotation''. Also we will always draw the graph $K^\bullet$ instead of the full graph $K$. For clarity, we keep the distinction between even vertices by black circles~$\bcol$ and odd vertices by white circles~$\wcol$.

We initialize the algorithm by adding all the $\MG$'s and $\XG$'es, that is,
$$
\Graph{
\GS 0.5/1:1/;
\GS 0.5/0:1/;
\GS 0.5/0.5:1/0.5,1/1.5;
\GV 0.5:1;
},\quad
\Graph{
\GS 0.5/1:1/;
\GS 0.5/0:1/;
\GS 0.5/0.75:1/0.57,1/1.43;
\GS 0.5/0.25:1/0.57,1/1.43;
\GV 0.5:1;
},\quad
\Graph{
\GS 0.5/1:1/;
\GS 0.5/0:1/;
\GS 0.5/0.85:1/0.65,1/1.35;
\GS 0.5/0.15:1/0.65,1/1.35;
\GS 0.5/0.5:1/0.5,1/1.5;
\GV 0.5:1;
},\quad
\Graph{
\GE 0.5/1:1/;
\GE 0.5/0:1/;
\GE 0.5/0.5:1/0.5,1/1.5;
\GW 0.5:1;
},\quad
\Graph{
\GE 0.5/1:1/;
\GE 0.5/0:1/;
\GE 0.5/0.75:1/0.57,1/1.43;
\GE 0.5/0.25:1/0.57,1/1.43;
\GW 0.5:1;
},\quad
\Graph{
\GE 0.5/1:1/;
\GE 0.5/0:1/;
\GE 0.5/0.85:1/0.65,1/1.35;
\GE 0.5/0.15:1/0.65,1/1.35;
\GE 0.5/0.5:1/0.5,1/1.5;
\GW 0.5:1;
}.
$$

Now, step (2A) adds the following graphs
$$
\Graph{
\GS 0.5/1:1/;
\GS 0.5/0:1/;
\GS 0.5/0.5:1/0.5;
\GE 0.5/1:2/;
\GE 0.5/0:2/;
\GE 0.5/0.5:2/1,2/2.5;
\GV 0.5:1;
\GW 0.5:2;
},\quad
\Graph{
\GS 0.5/0.5:1/0.5;
\GS 0.5/0.93:1/0.75,1/1.25;
\GS 0.5/0.07:1/0.75,1/1.25;
\GE 0.5/1:2/;
\GE 0.5/0:2/;
\GE 0.5/0.5:2/1,2/2.5;
\GW 0.5:2;
\GV 0.5:1;
},\quad
\Graph{
\GS 0.5/1:1/;
\GS 0.5/0:1/;
\GS 0.5/0.5:1/0.5;
\GE 0.5/0.5:2/1,2/2.5;
\GE 0.5/0.93:2/1.75,2/2.25;
\GE 0.5/0.07:2/1.75,2/2.25;
\GV 0.5:1;
\GW 0.5:2;
}.
$$
The step (2B) adds nothing at this point. We can apply again step (2A), namely to the first graph of the last three, to obtain
$$
\Graph{
\GS 0.5/1:1/;
\GS 0.5/0:1/;
\GS 0.5/0.5:1/0.5;
\GE 0.5/0:2/1.65,2/2.35;
\GE 0.5/0.5:2/1,2/3;
\GS 0.5/1:3/;
\GS 0.5/0:3/;
\GS 0.5/0.5:3/3.5;
\GV 0.5:1,3;
\GW 0.5:2;
},\quad
\Graph{
\GS 0.5/1:1/;
\GS 0.5/0:1/;
\GS 0.5/0.5:1/0.5;
\GE 0.5/1:2/;
\GE 0.5/0:2/;
\GE 0.5/0.5:2/1,2/3;
\GS 0.5/1:3/;
\GS 0.5/0:3/;
\GS 0.5/0.5:3/3.5;
\GV 0.5:1,3;
\GW 0.5:2;
},\quad
\Graph{
\GE 0.5/1:1/;
\GE 0.5/0:1/;
\GE 0.5/0.5:1/0.5;
\GS 0.5/0:2/1.65,2/2.35;
\GE 0.5/0.5:2/1,2/3;
\GE 0.5/1:3/;
\GE 0.5/0:3/;
\GE 0.5/0.5:3/3.5;
\GW 0.5:1,3;
\GV 0.5:2;
},\quad
\Graph{
\GE 0.5/1:1/;
\GE 0.5/0:1/;
\GE 0.5/0.5:1/0.5;
\GS 0.5/1:2/;
\GS 0.5/0:2/;
\GE 0.5/0.5:2/1,2/3;
\GE 0.5/1:3/;
\GE 0.5/0:3/;
\GE 0.5/0.5:3/3.5;
\GW 0.5:1,3;
\GV 0.5:2;
}.
$$
Finally, to the first and to the third graph of the last results, we can apply (2B) and in both cases we get the same result
$$
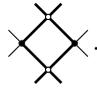
\Graph{
\GE 1.5/2:2/1.5,2/2.5;
\GE 0.5/1.5:1/2,3/2;
\GS 0.5/1:1/0.5,3/3.5;
\GS 0.5/0:1/0.5,3/3.5;
\GE -0.5/0.5:2/1,2/3;
\GE -0.5/-1:2/1.5,2/2.5;
\GW 1.5:2;
\GV 0.5:1,3;
\GW -0.5:2;
}.
$$
\end{ex}

\subsection{Proof of Theorem \ref{T.Cfin}}
\label{secc.proof}
In this subsection, we are going to show that Algorithm \ref{A} terminates and hence the sets $C(k,l)$ are finite. In order to do so, we can actually ignore the structure of the graphs and focus only on the vectors of input/output vertices $(\mathbf{a},\mathbf{b})$.

Let $\Sigma$ be a~countable alphabet partitioned into even and odd part $\Sigma=\Sigma_0\cup \Sigma_1$ (both parts countable, mutually disjoint). We define a~language (i.e.\ a~set of words) $L\subset\Sigma^*$ as follows. For any $a\in\Sigma$ and $k\in 2\N\setminus\{2\}$, we put $a^k\in L$. In addition, we define the following two families of production rules.

\begin{enumerate}
\item[(A)] Substitute any letter of a~word $w$ by $l$ copies, $l\in\{3,5,7,\dots\}$, of a~new one that has opposite parity and does not occur in $w$.
\item[(B)] Suppose there are two consecutive letters (alternatively the first and the last letter) in $w$ that are not equal, but have the same parity. Then substitute them by $l$ copies, $l\in\{2,4,6,\dots\}$, of a~new letter that has opposite parity and does not occur in $w$.
\end{enumerate}

The following example illustrates the derivation of some words in $L$. We consider lower case letters to be even and upper case letters to be odd.
$$\mathsf{aaaa}\xrightarrow[l=3]{\text{(A)}}
\mathsf{BBBaaa}\xrightarrow[l=5]{\text{(A)}}
\mathsf{BBBCCCCCaa}\xrightarrow[l=2]{\text{(B)}}
\mathsf{BBddCCCCaa}$$

Note that all the production rules make the given word longer except for (B) with $l=2$, which preserves the length. We say that a~word $w\in\Sigma^*$ is \emph{infinitely iterable} if one can iterate rule (B) with $l=2$ forever.

There are surely words, where one can never apply rule (B). For example, $\mathsf{aaaa}$ since all the letters are the same or $\mathsf{aBcD}$ since the parity alternates. It is not hard to think of examples, where one can use it only once such as $\mathsf{abbb}$. On the other hand, there are examples of words $w\in\Sigma^*$, where one can do the iteration forever such as
$$\mathsf{aabb}\to\mathsf{aCCb}\to\mathsf{DCCD}\to\mathsf{aaCD}\to\mathsf{aabb}\to\dots$$

We will say that a~word $w$ is \emph{strongly infinitely iterable} (s.i.i.) if one can iterate rule (B) for $l=2$ forever always taking two neighbouring letters and never the first and last one. If a~word is not infinitely iterable (i.i.), then it is not s.i.i. If a~word is not s.i.i., then none of its subwords is s.i.i.

\begin{lem}
\label{L.infit}
There is no infinitely iterable word in $L$.
\end{lem}
\begin{proof}
We will prove it by induction. For the sake of contradiction, suppose that $w\in L$ is i.i., but any shorter $w'\in L$ is not. Note that the words $a^k$, $k\in 2\N\setminus\{2\}$ are not i.i.; moreover, one cannot apply the rule (B) on them at all. Hence, without loss of generality, we can assume that $w$ was made by some production rule from some other word $w'$ that is not i.i. Equivalently, this means that $w'$ is strictly shorter and $w$ was not made by applying (B), $l=2$ (otherwise $w'$ would necessarily be i.i.).

Consequently, $w$ contains at least three identical consecutive letters. Without lost of generality, assume $w=aaav$, $a\in\Sigma_0$, $v\in\Sigma^*$, where those letters $a$ were made by the last operation. Undoing the last operation (and maybe doing some different one), we see that we have also $w'=bv\in L$ for some letter $b\in\Sigma$ (possibly different from $a$). Since $w'$ is shorter, it cannot be i.i.\ and hence $v$ is not s.i.i.

Since we assume that $w$ is i.i., we can denote by $(w_i)_{i\in\N}$ the corresponding sequence of iterations. Since $v$ is not s.i.i., there must be an index $i_1\in\N$ such that $w_{i_1-1}=aaav_{i_1-1}$ but $w_{i_1}=aaBBv_i$ or $w_{i_1}=Baav_iB$, where $B\in\Sigma_1$ and $v_i$ was made from $v_{i_1-1}$ by deleting the first resp.\ last letter. Since $v_{i_1-1}$ was made by some strong iterations from $v$ and $v_{i_1}$ is its subword, it also cannot be s.i.i. Hence, we can repeat the argument.

Choosing for example the first possibility, we must then reach some $i_2\in\N$ such that $w_{i_2-1}=aaBBv_{i_2-1}$, but $w_{i_2}=aaBccv_{i_2}$ or $w_{i_2}=CaBBv_{i_2}C$, where $c\in\Sigma_0$, $C\in\Sigma_1$. Again, $v_{i_2}$ cannot be s.s.i., so we can iterate this procedure.

Eventually, we end up with some $i_n\in\N$ such that $v_{i_n}$ is empty and, up to rotation, we have $w_{i_n}=a_0a_0a_1a_2\cdots a_{n-1}a_na_n$ such that the parity of the $a_i$'s alternates. Obviously all words of $L$ are of even length, so also $w_{i_n}$ is. Consequently, the parity of $a_0$ and $a_n$ is also different. But this means that one can no more apply the rule (B) on $w_{i_n}$. This is a~contradiction.
\end{proof}

\begin{proof}[Proof of Theorem \ref{T.Cfin}]
We are going to show that Algorithm \ref{A} terminates. Because of the restriction $k+l\le k_0$, the only way how the algorithm may not terminate is that there is some bilabelled graph $\KG\in C(k,l)$ such that one can iterate the step (2B) with $l=2$ on an appropriate rotation of the corresponding graph infinitely many times as this is the only possibility how to preserve the number of input/output vertices.

The rules (A) and (B) for the language $L$ exactly correspond to steps (2A) and (2B) of Algorithm \ref{A}. Hence (if we choose the alphabet $\Sigma$ appropriately), for any connected graph $\KG=(K,\mathbf{a},\mathbf{b})\in C(k,l)$ (that is any graph that is obtained in steps (1)--(4) of the algorithm), we have $w:=a_k\cdots a_1b_1\cdots b_l\in L$. If now one could iterate the step (2B) with $l=2$ on $\KG$ ad infinitum, it would mean that we can iterate step (B) with $l=2$ on the word $w$, but this is impossible as we just proved in Lemma~\ref{L.infit}
\end{proof}

\section{Concluding remarks and open problems}
\label{sec.open}

Let us mention some additional observations and concluding remarks here. In particular, we believe that this work opens wide possibilities for further research, so we mention a number of open problems in this section.

\subsection{Defining a~category structure on $C$}
\label{sec.Ccat}

We defined the collection of sets $C(k,l)$ by restricting to special graphs in $\Cat(k,l)$ in order to obtain sets that are finite while keeping the image under the functor $T^{T_\tau}$. Much more natural, however, is not to look for subsets of $\Cat(k,l)$, but for quotient sets. This allows us to keep the category structure.

\begin{defn}
We define an equivalence relation on the sets $\Cat(k,l)$ by taking the symmetric and transitive closure of the following. Graph $\KG$ is equivalent to $\KG'$ if $\KG'$ was made by (iv) a~two-path contraction, (v) by erasing exactly two edges between given two vertices or (vi) by erasing an isolated vertex that is not an input/output vertex. The quotient sets are denoted by $\bar\Cat(k,l)$ and form again a~category.
\end{defn}

\begin{rem}
As follows from Proposition \ref{P.FCCat}, the functors $\F_\tau$ and $T^{T_\tau}$ do not pass to $\bar\Cat(k,l)$. If $\KG'$ was made from $\KG$ by erasing a~pair of edges between two vertices, then we have $\F_\tau\KG=1/4\F_\tau\KG'$. If $\KG'$ was made from $\KG$ by erasing an isolated vertex of $K^\circ$, then $\F_\tau\KG=4\F_\tau\KG'$. There are two ways how to deal with this.

First possibility is to define $T^{T_\tau}$ on $\bar\Cat(k,l)$ by fixing some representative -- namely we can take the smallest one, which is a~graph from $C(k,l)$. Then $T^{T_\tau}$ is well defined, but it is not a~functor. The functorial property holds only up to a~multiplicative constant (similarly as in the case of Banica--Speicher easy quantum groups, cf.~\cite[Prop.~1.9]{BS09}). Second possibility is to introduce a~linear structure on $\Cat$ as in the following definition.
\end{rem}

\begin{defn}
Let $\Cat\nlin(k,l)$ be the vector space of formal linear combinations of elements in $\Cat(k,l)$. We extend the category operations to $\Cat\nlin(k,l)$ to define a~linear category. We define $\bar\Cat\nlin(k,l)$ to be a~quotient vector space of $\Cat\nlin(k,l)$ with respect to the relations
\begin{enumerate}\setcounter{enumi}{3}
\renewcommand{\theenumi}{\roman{enumi}}
\item $\KG=\KG'$ if $\KG'$ was made from $\KG$ by a~two-path contraction,
\item $\KG=1/N\,\KG'$ if $\KG'$ was made from $\KG$ by erasing a~pair of edges between two vertices
\item $\KG=N\,\KG'$ if $\KG'$ was made from $\KG$ by erasing a~vertex $v$ with $\tilde d_v=0$.
\end{enumerate}
\end{defn}

This defines an interesting diagrammatic category, which is maybe worth studying further. Let us suggest some questions in the following sections.

\subsection{Looking for a fibre functor for $\bar\Cat\nlin$} The mapping $T^{T_{\tau}}$ defines a fibre functor (i.e.\ a functor to the category of matrices) on $\bar\Cat\nlin$ for $N=4$. But the definition of $\bar\Cat\nlin$ works essentially for any complex number $N$. Therefore, we have the following natural question.

\begin{quest}
Are there some other fibre functors for $\bar\Cat\nlin$? (Considering maybe $N\in\{5,6,7,\dots\}$.)
\end{quest}

We are essentially looking for some pair of tensors interpreting the even and the odd vertices. Note an important fact that thanks to Proposition~\ref{P.unique} these tensors do not have to be permutation invariant in their indices.

\subsection{Functor injectivity, semisimplicity}

We were able to find finite subsets $C(k,l)\subset\Cat(k,l)$ describing the quantum group $D_4^+$. The natural question is now, whether those are the minimal subsets.

\begin{quest}
Is the set of intertwiners $\{T^{T_\tau}_\KG\mid\KG\in C(0,k)\}$ linearly independent for every $k\in\N_0$? Equivalently, is the functor $T^{T_\tau}$ injective on $\bar\Cat\flin$? That is, is the category $\bar\Cat\flin$ isomorphic to $\Cat_{D_4^+}$.
\end{quest}

It is known that such questions can actually be characterized within the diagram category itself without knowing the particular form of the corresponding functor. Hence, we can ask the same question for $\bar\Cat\nlin$ for arbitrary $N$ even though we do not know any fibre functor here yet:

\begin{quest}
Fix any $N\in\C$. Is the bilinear form $(\KG,\HG)\mapsto \KG^*\HG$, where $\KG,\HG\in\bar\Cat\nlin(0,k)$, non-degenerate for every $k\in\N_0$?
\end{quest}

Note also that this is equivalent to the Karoubi envelope of the corresponding category being semisimple. See e.g.\ \cite{Jun19,FM20} for more detailed discussion of such questions in case of partition categories.

\subsection{Law of characters and category isomorphisms}

We were  able to prove that the sets $C(k,l)$ are finite and we are able to generate them. However, we are not able to count them. This is also an important question:

\begin{quest}
Compute the number of elements $\# C(k,l)=\dim\bar\Cat\nlin(k,l)$.
\end{quest}

Note that if the functor $T^{T_\tau}$ is injective, answering this question gives us the dimensions of the fixed point spaces $\Cat_{D_4^+}(0,k)=\Mor(1,u^{\otimes k})$, which is a very important quantity. A related question then is to determine the law of the character $\chi=\sum_{i}u_{ii}$, see \cite{BS09,Wor87}.

Computing these numbers for small $k$ actually leads to a very exciting conjecture:

\begin{conj}
We have $\#C(0,2k)=C_k^2$, where $C_k$ are the Catalan numbers.
\end{conj}

Note that Catalan numbers are counting the non-crossing partitions as well as the non-crossing pair partitions. More precisely,
$$C_k=\# NC(0,k)=\# NC_2(0,2k),$$
where $NC$ is the category of all non-crossing partitions and $NC_2$ is the category of all non-crossing pair partitions (in the Banica--Speicher sense, without considering the linear structure). This leads to another even more exciting conjecture:

\begin{conj}
The category $C$ is isomorphic to the product category $NC_2\times NC_2$.
\end{conj}

Here, we consider the graph categories and partition categories without the linear structure. Generalizing this to the linear case should be straightforward, one just needs to tune the loop parameters $N$ appropriately.

Note that if such an isomorphism exists, then it has to map
$$\Gfour\mapsto\Labba\times\Laabb,\qquad\GWfour\mapsto\Laabb\times\Labba$$
or the other way around (up to scaling in the linear case). We checked on some examples that such a mapping indeed behaves well for small $k$, but we were not able to prove this in full generality.

Cartesian products of partition categories are studied in \cite{CW16}. An interesting fact is that the category $C$ is by definition generated by $\Gfour$ and $\GWfour$; in contrast, it is an open problem whether $NC_2\times NC_2$ is generated by $\Labba\times\Laabb$ and $\Laabb\times\Labba$.

Solving this conjecture would then help solving the problems that we presented before -- it would provide us a new fibre functor for $\bar\Cat\nlin$ and it would solve the problem about semisiplicity since the bilinear form associated to $NC_2$ is well understood.

\subsection{Questions concerning (free) Coxeter $D_N$}

There are also many open questions regarding our original motivation -- finding a~free analogue for Coxeter groups of type $D$. An obvious question is whether we can do the same also for different~$N$.

\begin{quest}
Is there a free analogue of Coxeter groups $D_N$ for $N>4$?
\end{quest}

One possible approach to answer this question might be to take the partition category associated to $D_N$, that is, $\Cat:=\langle\widehat{\singleton^{\otimes N}}\rangle\nlin$ and compute $\Cat^+:=\Cat\cap NC$. Then we can ask whether $\Cat^+$ is a non-crossing version of $\Cat$. That is, whether $\langle\Cat^+,\crosspart\rangle\nlin=\Cat$. This is not clear now as we are not able to compute this category explicitly.

\begin{quest}
\label{Q.PcatN}
Find an explicit description for the category $\langle\widehat{\singleton^{\otimes N}}\rangle\nlin$ for arbitrary $N\in\N$.
\end{quest}

We may also study the quantum group $D_4^+$ itself. Another question might be how meaningful this quantum group is as a~free analogue of $D_4$. Does it naturally appear in some applications? For instance, we can ask the following.

\begin{quest}
Is there a~finite graph $\Gamma$ such that its quantum automorphism group is isomorphic to $D_4^+$?
\end{quest}

\begin{rem}
The Coxeter groups of type $D$ of rank $N$ are symmetry groups of the so-called $N$-dimensional \emph{demihypercube} (created from the ordinary hypercube by ``taking only the odd vertices''). A natural candidate for the graph $\Gamma$ from the question above should be the 4-demihypercube then. However, there is an exception exactly for $N=4$. The 4-demihypercube is actually the complement of a graph consisting of four isolated segments. Hence, its automorphism group is the hyperoctahedral group $H_4=\Z_2\wr S_4$ and the quantum automorphism group is the free hyperoctahedral quantum group $H_4^+=\Z_2\wr_*S_4^+$. (Here $\wr_*$ denotes the \emph{free wreath product}, which describes the quantum automorphim group of $n$ copies of a given graph \cite{Bic04}.)

In addition, note that $D_N$ is also the automorphism group of the so-called $N$-dimensional \emph{folded hypercube}, but again only for $N>4$ \cite{Mir16}. ($N$-dimensional folded hypercube is created from the ordinary $N$-hypercube by identifying opposite vertices or, alternatively, from $(N-1)$-hypercube by connecting opposite vertices). The quantum automorphism group of the folded hypercube is $SO_N^{-1}$ for $N>4$ odd \cite{Sch20}. Using the same proof technique, one can actually obtain the same result also for the $N$-demihypercube. For $N=4$, the folded hypercube graph coincides with the full bipartite graph $K_{4,4}$, which is the complement of $K_4\sqcup K_4$. Hence, its automorphism group is $S_4\wr\Z_2$ and its quantum automorphism group is $S_4^+\wr_*\Z_2$.
\end{rem}

\appendix

\section{The way of drawing the elements of $C$ is unique}

As a side remark, we would like to mention a result that does not tell us much about the category $\Cat_{D_4^+}$, but might be useful when studying some of the open problems sketched in Section~\ref{sec.open}.

We prove namely that the way of drawing the bilabelled graphs in $C$ is unique. This may be good to know when studying $\bar\Cat\nlin$ as an abstract category and, in particular, when looking for its fibre functors. 

\begin{prop}
\label{P.unique}
For any $\KG=(K,\mathbf{a},\mathbf{b})\in C(k,l)$, there is a~unique planar drawing of $\KG$. More precisely, there is (up to orientation and choice of the unbounded face) unique planar embedding of the graph $K^\odot$ and the orientation is fixed by fixing the orientation of the enveloping cycle.
\end{prop}
\begin{proof}
First of all, we need to assume that $K$ is connected. For the unconnected case, we can then use induction. If $K$ is not connected, then, up to rotation, we have $\KG=\HG_1\otimes\HG_2$. From induction, $\HG_1$ and $\HG_2$ have unique drawing and finally one can see that the only possible drawing of $\KG$ is to put $\HG_1$ and $\HG_2$ side by side.

We also need to assume that $|V(K^\bullet)|\ge 2$. But obviously all the graphs $\nullG$, $\MG^{k,l}$, $\XG^{k,l}$ have a~unique planar drawing being the only instances of $\KG$ with $|V(K^\bullet)|\le1$.

We will use the Whitney's theorem saying that every 3-connected graph has a~unique planar drawing. However, we cannot apply it to the graph $K^\odot$ directly since it is rarely 3-connected. Instead, we construct the graph $K^{\bullet\odot}$ from $K^\bullet$ the same way as $K^\odot$ is constructed from $K$. That is, add the \emph{enveloping cycle} $\alpha_k^\bullet,\dots,\alpha_1^\bullet,\beta_1^\bullet,\dots,\beta_k^\bullet$ and edges $\{\alpha_i^\bullet,a_i^\bullet\}$ and $\{\beta_i^\bullet,b_i^\bullet\}$ and finally one last vertex connected to all the vertices of the enveloping cycle. Now, one can see that $K^{\bullet\odot}$ differs from $K^\odot$ just by the fact that some edges in $K^{\bullet\odot}$ are two-paths in $K^\odot$, but anyway $K^{\bullet\odot}$ has a~unique planar drawing if and only if $K^\odot$ has.

So, take any $\KG\in C(k,l)$ such that $K$ (equivalently $K^\bullet$) is connected and $|V(K^\bullet)|\ge2$. We are going to prove that $K^{\bullet\odot}$ is 3-connected. To illustrate the idea, let us first prove that it is 2-connected. For the sake of contradiction, suppose that $v$ is a~separating vertex in $K^{\bullet\odot}$. Since $K^\bullet$ is connected, we must have $v\in K^\bullet$. Take some connected component $H^\bullet$ of $K^{\bullet\odot}\setminus\{v\}$ that does not contain the enveloping cycle and denote by $m$ the number of edges between $v$ and $H'$. This component induces a~bilabelled graph $\HG=(H,\emptyset,\mathbf{c})\in C$ with $c_1^\bullet,\dots,c_m^\bullet$ being the neighbours of $v$ in $H^\bullet$. (If $v$ is odd then $H=H^\bullet$ and $c_i=c_i^\bullet$ are the neighbours of $v$. If $v$ is even, then we have to introduce new vertices $c_1,\dots,c_m$ that split the edges between $v$ and $H^\bullet$.) Now, $H^\bullet$ must contain at least two vertices since if there was a~single vertex $u$, then we would have $\tilde d_u=1$ in $\KG$, which violates~(ii). Consequently, we can use Lemma \ref{L.induction}, to find a~vertex $u$ in $H^\bullet$ with $d_u^\bullet\le 2$ in $\HG$. This means that $u$ appears at least twice in $(c_1^\bullet,\dots,c_m^\bullet)$ and hence there is a~double edge between $u$ and $v$ in $\KG$, which is a~contradiction.

The proof of 3-connectedness is similar just slightly more complicated. Let $\{v,w\}$ be the separating pair. Suppose first that both are actually elements of $K^\bullet$. Again, we find a~connected component $H^\bullet$ of $K^{\bullet\odot}\setminus\{v,w\}$ not containing the enveloping cycle and define the corresponding bilabelled graph $\HG=(H,\emptyset,\mathbf{c})$. Some of the vertices $c_i^\bullet$ are connected to $v$ in $K^\bullet$ and some are connected to $w$. Now, one can modify the proof of Lemma \ref{L.induction} to show that there actually must be at least two vertices $u_1,u_2$ in $\HG$ such that $d_{u_i}^\bullet\le 2$ in $\HG$ and $H^\bullet\setminus\{u_i\}$ is connected. So, according to Lemma \ref{L.consec}, both $u_1$, $u_2$ must appear at least twice and consecutively in $(c_1^\bullet,\dots,c_m^\bullet)$. From this, it follows that at least of those two vertices must have a~double edge with either $v$ or $w$.

Secondly, assume $\{v,\alpha\}$ is a~separating pair, where $v\in K^\bullet$ and $\alpha$ is an element of the enveloping cycle adjacent to some vertex $a\in K^\bullet$. Let $H^\bullet$ be the component of $K^{\bullet\odot}\setminus\{v,\alpha\}$ not containing the enveloping cycle but containing $a$ (such a~component exists, otherwise $v$ is a~separating vertex and $K^{\bullet\odot}$ is not even 2-connected, which was already excluded above). Now, again there are at least two vertices $u$ in $H^\bullet$ such that $d_u^\bullet\le 2$ in $\HG$, so one of them is not $a$ and must be connected with a~multiple edge to $v$ in $\KG$.

A pair of vertices of the enveloping cycle can be separating only if $k+l=2$, which happens only for $\MG^{1,1}$ and this case was already discussed above as it does not satisfy $|V(K^\bullet)|\ge 2$.

Finally, take $\KG=(K,\mathbf{a},\mathbf{b})\in C(k,l)$ and let us show that the orientation of the drawing of $K^\odot$ is uniquely determined by the orientation of the enveloping cycle. So, suppose that reflecting a~drawing of $K^\odot$ preserves the order of the enveloping cycle. In other words, define the word $w:=a_k\cdots a_1b_1\cdots b_l$ and suppose that the reflection of $w$ (reading the word backwards) is a~rotation of $w$. For now, suppose also that $K$ is connected. We prove that $K$ consists of a~single path and hence its drawing is equivalent to the reflection.

Without loss of generality, we can assume $w=x^ma_1^{i_1}\cdots a_k^{i_k}y^na_k^{i_k}\cdots a_1^{i_1}$, where $a_i\neq a_{i+1}$. Recall the algorithm generating all elements of $C$. The word $w$ must have been constructed by applying rules (A) and (B) on some initial word $w_0=a^k$, $k\in\nobreak2\N\setminus\{2\}$. It is enough to show that we only used rule (A) as this leads to the graph consisting of a~single path. Note that the last operation must have created either $x^m$ or $y^n$, without loss of generality, assume the second possibility. Let us undo the last operation that created $w$. We claim that it must have been the operation (A). If it was (B), then the predecessor of $w$ was of the form $w'=x^ma_1^{i_1}\cdots a_k^{i_k}uva_k^{i_k}\cdots a_1^{i_1}$. If neither of $u$ and $v$ equals $a_k$, then it is a~contradiction with the generalization of Lemma \ref{L.induction} since there is only one letter appearing consecutively at least twice in $w'$. If, say $u=a_k$, then it is a~contradiction with the 3-connectivity of the associated graph we proved above since removing the vertex $a_k$ and the element of the enveloping cycle corresponding to $v$ isolates the vertex $v$.

So, $w$ was created from some $w'$ by applying rule (A). But this means that $w'$ was also a~``palindrom up to rotation'', so by induction $w$ was created only by repeated application of rule (A). 

If $K$ is not connected, we again use induction. We have $\KG=\HG_1\otimes\HG_2$ (up to rotation) and by induction hypothesis both $\HG_1$ and $\HG_2$ have planar drawings that are equivalent to their reflections. Then this must hold also for $\KG$.
\end{proof}

\bibliographystyle{halpha}
\bibliography{mybase}

\end{document}